\documentclass[11pt]{article}
\usepackage[ansinew]{inputenc}
\usepackage{graphicx}
\usepackage{color}
\usepackage{hyperref}

\usepackage{enumerate,latexsym}
\usepackage{latexsym}
\usepackage{amsmath,amssymb}
\usepackage{graphicx}
\usepackage{soul}
\usepackage{times}
\usepackage[dvipsnames]{xcolor}

\newfont{\bb}{msbm10 at 11pt}
\newfont{\bbsmall}{msbm8 at 8pt}

\def\cL{\cal{L}}
\def\ov{\overline}
\def\rth{\mathbb{R}^3}
\def\R{\mathbb{R}}
\def\B{\mathbb{B}}
\def\N{\mathbb{N}}

\def\C{\mathbb{C}}

\def\D{\mathbb{D}}
\def\Pe{\mathbb{P}}
\def\esf{\mathbb{S}}

\def\cM{\mathcal{M}}
\def\cD{\mathcal{D}}
\def\cC{\mathcal{C}}
\def\cH{\mathcal{H}}

\newcommand{\wh}{\widehat}
\newcommand{\Int}{\mbox{\rm Int}}

\def\a{{\alpha}}
\def\lc{{\cal L}}

\def\g{{\gamma}}
\def\G{{\Gamma}}

\def\de{{\delta}}
\def\be{{\beta}}
\def\ve{{\varepsilon}}

\newtheorem{theorem}{Theorem}[section]
\newtheorem{lemma}[theorem]{Lemma}
\newtheorem{proposition}[theorem]{Proposition}

\newtheorem{remark}[theorem]{Remark}

\newtheorem{definition}[theorem]{Definition}
\newtheorem{conjecture}[theorem]{Conjecture}
\newtheorem{assertion}[theorem]{Assertion}

\newenvironment{proof}{\smallskip\noindent{\it Proof.}\hskip \labelsep}
                          {\hfill\penalty10000\raisebox{-.09em}{$\Box$}\par\medskip}

\newcommand{\ed}{\end{document}}
\usepackage{pdfsync}

\begin{document}
 \begin{title}
{Bounds on the topology and index of  minimal surfaces }
\end{title}

\vskip .5in

\begin{author}
{William H. Meeks III\thanks{This material is based upon
 work for the NSF under Award No. DMS -
  1309236. Any opinions, findings, and conclusions or recommendations
 expressed in this publication are those of the authors and do not
 necessarily reflect the views of the NSF.}
 \and Joaqu\'\i n P\' erez
\and Antonio Ros\thanks{Research of the second and third authors partially supported  by
MINECO/FEDER grants no. MTM2014-52368-P and MTM2017-89677-P.}}

\end{author}

\maketitle

\begin{abstract} We prove that for every nonnegative integer $g$, there
exists a bound on the number of ends of a complete, embedded minimal surface
$M$ in $ \R^3$ of genus $g$ and finite topology. This bound on the finite number of
ends when $M$ has at least two ends  implies that $M$ has finite stability index
which is  bounded by a constant that only depends on its genus.
\vspace{.17cm}


 \end{abstract}

\section{Introduction}
\label{secintrod} Let ${\cal M}$ be the space of connected, properly
embedded minimal surfaces in $\R^3$. The focus of this paper is to
prove the existence of an upper bound on the number of ends for a
surface $M\in {\cal M}$ having finite topology, solely in terms of
the genus of $M$. In the case that $M$ has more than one end, this topological bound
also produces an upper bound for the index of stability of $M$.

There are three classical
conjectures that attempt to describe the topological types of the surfaces
occurring in ${\cal M}$.

\begin{conjecture}[Finite Topology Conjecture I, Hoffman-Meeks] \label{c1.1}
A non-compact orientable surface with finite genus $g$ and a finite
number of ends $k>2$ occurs as the topological type of an example in
${\cal M}$ if and only if $k \leq g+2$.  A minimal surface in ${\cal
M}$ with finite genus and two ends has genus zero and is a catenoid.
\end{conjecture}

\begin{conjecture}[Finite Topology Conjecture II, Meeks-Rosenberg] \label{c1.2}
For every positive integer $g$, there exists
a $\Sigma _g\in {\cal M}$ with one end
and genus $g$, which is unique up to congruences
and homotheties. Furthermore, if $g=0$,
such a $\Sigma _g$ is a plane or a helicoid.
\end{conjecture}

\begin{conjecture}[Infinite Topology Conjecture, Meeks]  \label{c1.3}
\mbox{}\\
A non-compact orientable surface of infinite topology occurs as the
topological type of an example in ${\cal M}$ if and only if it has
at most two limit ends, and when it has one limit end, then its
limit end has infinite genus.
\end{conjecture}

For a detailed discussion of these conjectures and related results,
we refer the reader to the survey by Meeks and P\'erez~\cite{mpe2}.
However, we make here
a few brief comments on what is known concerning
these conjectures and which will be used in the proof of the main
theorem of this paper.

Regarding Conjecture~\ref{c1.1}, a theorem by Collin~\cite{col1}
states that
if $M\in{\cal M}$ has finite topology and at least two ends, then
$M$ has finite total Gaussian curvature. This result implies that
such a surface is conformally a compact Riemann surface
$\overline{M}$ punctured in a finite number of points and $M$ can be
defined in terms of meromorphic data on its conformal
compactification $\overline{M}$ (Osserman~\cite{os1}). Collin's Theorem reduces the
question of finding topological obstructions for surfaces in~${\cal M}$ of finite
topology and more than one end to the question of finding topological
obstructions for complete, embedded minimal surfaces of finite total
curvature in~$\R^3$. For example, if $M$ is a complete, embedded
minimal surface in $\R^3$ with finite total curvature, genus $g$ and
$k$ ends, then $M$ is properly embedded in $\R^3$ and the
Jorge-Meeks formula~\cite{jm1} calculates its total curvature to be
$-4\pi (g+k-1)$. The first topological obstructions for complete,
embedded minimal surfaces $M$ of finite total curvature were given
by Jorge and Meeks~\cite{jm1}, who proved that if $M$ has genus
zero, then $M$ does not have $3,4$ or $5$ ends. Later this result
was generalized by L\'opez and Ros~\cite{lor1} who proved that the
plane and the catenoid are the only genus-zero minimal surfaces of
finite total curvature in ${\cal M}$. At about the same time,
Schoen~\cite{sc1} proved that a complete, embedded minimal surface
of finite total curvature and two ends must be a catenoid.

The existence theory for properly embedded minimal surfaces with
finite total curvature was begun by Costa~\cite{co2} and by Hoffman and
Meeks~\cite{hm7}, with important theoretical advances by Kapouleas~\cite{kap4}
and  Traizet~\cite{tra1}.  A paper by Weber and
Wolf~\cite{ww1} makes the existence assertion in Conjecture~\ref{c1.1}
 seem likely to hold, although their results actually fall short of giving a
proof of embeddedness for their examples.

Concerning Conjecture~\ref{c1.2}, a theorem by Meeks and Rosenberg~\cite{mr8}
states that the plane and the helicoid are the only properly
embedded, simply connected minimal surfaces in $\R^3$. They also
claimed that if $M\in {\cal M}$ has finite positive  genus and just
one end, then it is asymptotic to the end of the helicoid and can be
defined analytically in terms of meromorphic data on its conformal
completion, which is a closed Riemann surface; the proof of this result
was given
by Bernstein and Breiner~\cite{bb2} and a  related
more general result was proved by
Meeks and P\'erez~\cite{mpe3}. These theoretical
results together with theorems developed by Weber and
Traizet~\cite{tw1}, by Hoffman and White~\cite{hofw1}
and by Hoffman,
Weber and Wolf~\cite{hweb1} provide the theory on which a proof of
the uniqueness part of Conjecture~\ref{c1.2} might be based.
The existence part of Conjecture~\ref{c1.2} has been
recently solved by Hoffman,
Traizet and White~\cite{htw1}.

In relation to Conjecture~\ref{c1.3},
there are two important topological obstructions for
surfaces in ${\cal M}$ with infinite topology:
Collin, Kusner, Meeks and Rosenberg~\cite{ckmr1} proved that an
example in ${\cal M}$ cannot have more than two limit ends,
and in~\cite{mpr4}
we proved that an example in ${\cal M}$ with one limit end cannot
have finite genus. This last result depends on Colding-Minicozzi
theory as well as on our previous
results in~\cite{mpr3} where we presented
a descriptive theorem for minimal surfaces in ${\cal M}$ with two limit
ends and finite genus. The study of two limit end minimal surfaces is motivated by
a one-parameter family of periodic examples of genus zero discovered by
Riemann~\cite{ri1}, which are defined in terms of elliptic functions on rectangular
elliptic curves. In~\cite{mpr6} we showed
that if $M\in \cM$ has  genus zero and infinite topology, then $M$ is one
of the Riemann minimal examples, and if $M$ has finite genus and infinite topology,
then  each of its two  limit ends
are in a natural sense asymptotic to the end of a Riemann minimal example.
It is worth mentioning that Hauswirth and Pacard~\cite{hauP1} have
produced examples in ${\cal M}$ with finite genus and infinite topology.

{\it A priori,} one procedure to obtain surfaces in ${\cal M}$ with finite genus and
infinite topology might be to take  limits of sequences of
finite total curvature examples in ${\cal M}$ with a bound
on their genus but with a strictly increasing number of ends. Our results
in~\cite{mpr11,mpr3,mpr4,mpr10} are crucial in
understanding that such sequences cannot exist,
and they will lead us to
a proof of the
following main theorem of this manuscript.

\begin{theorem}
\label{thm1} A properly embedded minimal surface in $\R^3$ with finite topology has
a bound on the number of its ends that only depends on its genus.
\end{theorem}

Colding and Minicozzi~\cite{cm35} proved that a
complete, embedded minimal surface of finite topology in $\R^3$ is
properly embedded. In particular,
the conclusion of Theorem~\ref{thm1} remains valid if we weaken the
hypothesis of properness to the hypothesis of completeness.

By a theorem of Fischer-Colbrie~\cite{fi1}, a complete, immersed, orientable minimal
surface $M$ in $\R^3$ has finite index of stability if and only if it has finite total
curvature. The index of such an $M$ is equal to the index of the Schr\"{o}dinger
operator $L=\Delta +\| \nabla N\| ^2$ associated to the meromorphic extension of the
Gauss map $N$ of $M$ to the compactification of $M$ by attaching its ends.
Grigor'yan, Netrusov and Yau~\cite{gny1}
made an in depth study of the relation between the
degree of the Gauss map and the index of a
complete minimal surface of finite total curvature.
In particular, they proved that the index of a
complete, embedded minimal surface with $k$ ends is bounded from below
 by $k-1$.
On the other hand, Tysk~\cite{ty} proved
that the stability index of $L$ can be explicitly
bounded from above in terms of the degree of $N$.
By the Jorge-Meeks formula for such an
embedded $M$, the degree of $N$ equals $g+k-1$,
where $g$ is the genus and $k$ is the
number of ends. Hence by Theorem~\ref{thm1}, if $g$ is
fixed, then $k$ is bounded for an
embedded $M$. Thus, one obtains the following
consequence to Theorem~\ref{thm1}.

\begin{theorem}
\label{thm2} If $M\subset \R^3$ is a complete, connected, embedded
minimal surface with finite index of stability, then
the index of $M$ can be bounded
by a constant that only depends on its (finite) genus.
In the case of genus
zero, the surface $M$ is a plane or catenoid,  and so this  index
upper bound is $1$.
\end{theorem}

For any integer $k\geq 2$, the $k$-noid defined by
Jorge and Meeks~\cite{jm1} has genus zero, $k$
catenoid type ends and index $2k-3$ (Montiel-Ros~\cite{mro1}
and Ejiri-Kotani~\cite{ek2}). Also,
there exist examples of complete, immersed minimal surfaces of genus zero with a
finite number of parallel catenoidal ends
but which have arbitrarily large index of stability. These
examples demonstrate the necessity  of the embeddedness hypothesis in
Theorem~\ref{thm2}.

The proof of Theorem~\ref{thm1} depends heavily
on results developed in our previous
papers~\cite{mpr11,mpr3,mpr4,mpr10}. These papers,
as well as the present one, rely
on a series of deep works by Colding and
Minicozzi~\cite{cm21,cm22,cm24,cm23,cm25} in which
they describe the local geometry of a complete, embedded minimal
surface in a Riemannian three-manifold, where there
is a local bound on the genus of the surface.

\section{Preliminaries.}
\label{secprelim}
Throughout the paper,  we will denote by $\B (x,r)$ the open ball in
$\R^3$ with center $x\in \R^3$ and radius $r>0$,
and by $\overline{\B }(x,r)$ its closure. When $x$ is the origin,
we will simply write $\B (r),\overline{\B }(r)$
respectively for these particular open and closed balls.
We let $\D (r)$ denote the open disk of radius $r$ centered at the origin in $\R^2$.
For a surface $\Sigma \subset \R^3$, $K_{\Sigma }$ will
denote its Gaussian curvature function.

\begin{definition}
{\rm
Let $A$ be an open subset of $\rth$. A sequence of surfaces $\{ \Sigma _n\} _n$ in  $A$
 is said to have {\it locally bounded curvature in $A$}, if for every compact
 ball $B\subset A$, the sequence of functions $\{ K_{\Sigma _n\cap B}\} _n$ is uniformly
 bounded. A sequence $\{ \Sigma _n\} _n$ of
properly embedded surfaces in  $A$ is called {\it
locally simply connected}, if for every $q\in A$,
there exist $\ve _q>0$ and $n_q\in \N $ such that for $n>n_q$,  the
components of $\Sigma _n\cap  \B (q,\ve _q)$
are disks with their boundaries in the boundary
 of $ \B (q,\ve _q)$.

 Proposition 1.1 in Colding and Minicozzi~\cite{cm35}
ensures that the property that a sequence of embedded minimal surfaces $\{ M(n)\} _n$ is
locally simply connected in $A$ is equivalent to the
property that $\{ M(n)\} _n$ has {\it locally positive injectivity
radius in $A$}, in the sense that for every $q\in A$,
there exist $\ve _q>0$ and $n_q\in \N $ such that for $n>n_q$, the
restriction to  $M(n)\cap \B (q,\ve _q)$ of the
injectivity radius function $I_{M(n)}$ of $M(n)$ is
greater than a positive constant independent of $n$.
In particular, if the $M(n)$ also have boundary, then for
any $p\in A$ there exist $\de _p>0$ and $n_p\in \N$ such that
$\partial M(n)\cap \B( p,\de _p)=\varnothing$ for all $n>n_p$, i.e.,
points in the boundary of
$M(n)$ must eventually diverge in space or converge to a subset of $\R^3\setminus A$.

We will call a sequence of surfaces $\{ \Sigma _n\}_n$ in $\rth$
 {\it uniformly locally simply connected},
if there exists $\ve >0$ such that for every $x\in \R^3$,
$\Sigma_n\cap \B (x,\ve )$ consists
of disks with boundary in $\partial \B (x, \ve)$ for all $n$
sufficiently large (depending on $x$).
}
\end{definition}

In the proof of our main Theorem~\ref{thm1} we will
use repeatedly the following statement, which is an application of
Theorem~1.6 in~\cite{mpr11}.

\begin{theorem}
\label{structurethm}
Suppose $W$ is a closed countable subset of $\R^3$. Let
$\{M_n\}_n$ be a sequence of smooth, connected, properly embedded minimal surfaces in $A=\R^3\setminus W$
such that:
\begin{enumerate}[i.]
\item $\{M_n\}_n$
has locally positive injectivity
radius  in $A$.
\item For each $n\in \N$, $M_n$ has compact boundary (possibly empty) and  genus
at most $g\in \N\cup\{0\}$, for some $g$ independent of $n$.
\end{enumerate} Then, after replacing $\{M_n\}_n$ by a
subsequence and composing the surfaces with a fixed
 rotation, there exists a minimal lamination $\cL$ of $A$ and a closed subset
${\cal S}({\cal L})\subset {\cal L}$  such that
$\{M_n\}_n$ converges $C^\a$, for all $\a \in (0,1)$, on compact subsets
of $A\setminus S({\cal L})$ to ${\cal L}$; here ${\cal S}(\cal L)$ is the singular
set of convergence\footnote{This means that for every $x\in S({\cal L})$
and for all $r>0$, we have $\limsup |K_{M_n\cap \B (x,r)}|=\infty $.}
of the $M_n$ to ${\cal L}$.

Furthermore:
\begin{enumerate}
\item The closure
$\overline{\lc}$ of $\lc$ in $\R^3$ has the structure
of a minimal lamination of $\R^3$.
\item \label{itt2} If $S({\cal L})\neq \mbox{\rm \O}$,
then the convergence of $\{M_n\}_n$ to  $\ov{\cal L}$
has the structure of a horizontal limiting parking garage structure
in the following sense:
\begin{enumerate}[2.1.]
\item  $\ov{\cal L}$  is a foliation of $\R^3$ by horizontal
planes and $\ov{S({\cal L})}$ consists of one or two
 vertical straight lines
(called columns of the limiting parking garage structure).
\item As $n \to \infty$, a pair of highly sheeted
multivalued graphs are forming inside $M_n$ around each
of the lines in $\ov{S({\cal L})}$, and if  $\ov{S({\cal L})}$ consists of two lines,
 then these pairs of
multivalued graphs inside the $M_n$ around different lines are
oppositely handed.
\end{enumerate}
\item \label{it5}
If  $\ov{\cL}$ contains
a non-flat leaf, then $S({\cal L})=\varnothing$ and
$\ov{\cal L}$ consists of a single leaf $L_1$,
which is properly embedded in $\R^3$ and the genus of $L_1$ is
at most $g$. Furthermore, 
$\{ M_n\} _n$ converges smoothly on compact sets in $\rth$ to $L_1$ with multiplicity~$1$ and one
of the following three cases holds for $L_1$.
\begin{enumerate}
\item $L_1$ has one end  and it is asymptotic to a helicoid. 
\item $L_1$ has non-zero finite total curvature.
\item $L_1$ has two limit ends.
    \end{enumerate}
\end{enumerate}
\end{theorem}
\begin{remark}
		\label{remark1}
{\rm
	Before giving the proof of Theorem~\ref{structurethm}, we will illustrate its statement with some examples.
\begin{enumerate}
	\item The limit of homothetic shrinkings $M_n=\frac{1}{n}C$ of a vertical catenoid $C=\{ \cosh^2(x^2+y^2)=z^2\} $
	is a particular case of Theorem~\ref{structurethm} with  $W=\{ \vec{0}\} $ and $g=0$.
	In this example, $\mathcal{L}$ is the punctured plane $\{ z=0\} \setminus \{ \vec{0}\}$ and
	$S(\mathcal{L})=\varnothing$.
	\item The classical Riemann minimal examples $R_t$, $t>0$, form a 1-parameter family of properly embedded,
	singly-periodic minimal surfaces with genus zero and
	infinitely many planar ends asymptotic to horizontal planes.
	As a particular case of Theorem~\ref{structurethm}, one can
	consider the limit of the $R_t$ when the flux vector of
	$R_t$ along a compact horizontal section converges to $(2,0,0)$. In this example, $W=\varnothing$, $g=0$
	and item 2.2 of Theorem~\ref{structurethm} holds.
	\item Theorem~0.9 in Colding-Minicozzi~\cite{cm25}
	can be viewed as a particular case of
	Theorem~\ref{structurethm} when $\{ M_n\} _n$ is a locally
	simply connected sequence in $\R^3$ of compact planar domains with $\partial M_n\subset \partial \B (R_n)$ and
	$R_n\to \infty $ (hence $W=\varnothing$ and $g=0$), with the
	additional assumptions:
	\begin{itemize}
		\item $\sup |K_{M_n\cap \B(y,r)}|\to \infty $ for some $y\in \R^3$ and for all $r>0$
		(hypothesis (0.2) in~\cite{cm25}). This implies the $S(\mathcal{L})\neq \varnothing $, and thus, item~2 of
		Theorem~\ref{structurethm} holds.
		\item There exists $R>0$ such that each $M_n$ intersects $\B (R)$ in a component that is not a disk
		(hypothesis (0.5) in~\cite{cm25}). This allows us to discard the case of a single column for the limiting
		parking garage structure in item~2.1 of Theorem~\ref{structurethm}, and item~2.2 holds.
	\end{itemize}
	\item Take a properly embedded minimal surface $M\subset \R^3$ with infinite total curvature. By the Dynamics Theorem
	(Theorem~2 in~\cite{mpr20}), the set $D_1(M)$ of non-flat properly embedded minimal surfaces $\Sigma \subset \R^3$ which are
	obtained as $C^2$-limits of a divergent sequence $\{ \lambda _n(M-p_n)\}_n$	of dilations of $M$ (i.e., the translational
	part $p_n\in \R^3$ of the dilations diverges) such that $\vec{0}\in \Sigma$, $|K_{\Sigma }|\leq 1$ on $\Sigma $ and
	$|K_{\Sigma}|(\vec{0})=1$, is non-empty. Given $\Sigma \in D_1(M)$ obtained as the limit of $M_n=\lambda _n(M-p_n)$, the fact
	that $\Sigma$ has locally positive injectivity radius in $\R^3$ implies that $\{ M_n\}_n$ also has locally positive injectivity
	radius in $\R^3$ (i.e., $W=\varnothing$ in Theorem~\ref{structurethm}). Now assume that $\Sigma $ has finite genus $g$. After
	possibly replacing $M_n$ by $[\lambda _n(M-p_n)]\cap \B(R_n)$ for an appropriate sequence of radii
	$R_n\to \infty $, we can assume
	that hypothesis ii of Theorem~\ref{structurethm} holds for $\{ M_n\}_n$. In this example, $\mathcal{L}=\Sigma$,
	$S(\mathcal{L})=\varnothing$ and item~3 of Theorem~\ref{structurethm} holds.
\end{enumerate}
}
\end{remark}
\begin{proof}[of Theorem~\ref{structurethm}]
The existence of the minimal lamination ${\cal L}$ of
$A$ such that
the $M_n$ converge to ${\cal L}$  (after passing to a subsequence)
in $A\setminus S({\cal L})$ follows directly from the main statement of
Theorem~1.5 in~\cite{mpr11}; note that
singularities of ${\cal L}$ are ruled out in our
setting by item~7.1 of Theorem~1.5 in~\cite{mpr11}
because the $M_n$ have uniformly bounded genus.
The same argument using  item~7.1 of Theorem~1. 5 in~\cite{mpr11} ensures that item~1 of
Theorem~\ref{structurethm} holds.

Now assume that $S({\cal L})\neq \varnothing$ and we will prove that  item~2 holds.
To accomplish this, we first define a sequence of auxiliary compact minimal surfaces.
For each $k\in \N$, choose an $R_k\in (k,k+1)$ such that the sphere
$\partial \B(R_k)$ is disjoint from $W$,
which is possible since $W$ is a countable set.  Since
$W \cap \ov{\B}(k+1)$ is a compact set, then $\partial \B(R_k)$ is at a positive distance $2d_k$
from $W$; define $W_k=W\cap \B(R_k)$ and for any $\ve>0$ let $W_k(\ve)$
be the open $\ve$-neighborhood of $W_k$ in $\rth$.
As the sequence $\{M_n\}_n$ is locally simply connected
 in $A$ and $W$ is at a positive
 distance from $\partial \B(R_k)$, then for $k$ fixed and $n$ sufficiently large,
 $\partial M_n$ is disjoint
 from $\partial \B(R_k)$ and by Sard's Theorem, we can also assume that
$R_k$ is chosen so that $\partial \B(R_k)$  is transverse to all of the surfaces $M_n$.
Also by Sard's Theorem, there exist smooth, compact,
possibly disconnected three-dimensional domains $\Delta_{n,k}\subset \B(R_k)$
satisfying $W_k\subset \Int(\Delta_{n,k})\subset W_k(d_k/n)$ and
such that $\partial \Delta_{n,k}$ is transverse to $M_n$ for all $k$ and for $n$ sufficiently large.
It follows that
$M_{n,k}=[M_n \setminus  \Int(\Delta_{n,k})] \cap \ov{\B}(R_k)$ is a  doubly indexed sequence
of smooth compact surfaces such that, after
replacing $M_n$ by a subsequence, $\{M_{k,k}\}_k$ is a sequence of
smooth, compact embedded minimal surfaces that
satisfy the hypotheses of Theorem~\ref{structurethm}. By construction,
the sequence $\{M_{k,k}\}_k$  also converges to $\cL$ with
the same singular set of convergence $S(\cL)$.
However, since the  minimal
surfaces in the sequence $\{M_{k,k}\}_k$ are compact with  genus
at most $g$, then item~7.3 in Theorem~1.5 of~\cite{mpr11}
implies that  item~\ref{itt2} of Theorem~\ref{structurethm}
holds.

Finally we demonstrate item~\ref{it5} of Theorem~\ref{structurethm}.
Suppose that the non-singular minimal lamination $\overline{\cal L}$ contains a non-flat
leaf $L_1$. By item~6 of Theorem~1.5 in~\cite{mpr11}, $L_1$ is properly
embedded in a simply connected
domain in $\rth$, which implies that  it is two-sided. In particular,
$L_1$ is not stable and thus, the convergence of portions of the
$\{M_{k,k}\}_k$ to $L_1$ must have multiplicity $1$ (see e.g., Lemma~3 in~\cite{mpr20}).
This last property and a standard curve lifting argument ensure that the
genus of $L_1$ is at most $g$.
By item~6 of Theorem~1.5 in~\cite{mpr11}, $L_1$ is proper in $\R^3$ and it is
the unique leaf of $\overline{\cal L}$ (that is, possibility~6.1 in that
theorem holds, because case 6.2 in the same
theorem cannot occur due to the finiteness of the genus of $L_1$). In particular, $S({\cal L})
=\varnothing$ since through every point in $S({\cal L})$ there passes a planar
leaf of $\overline{\cal L}$ by item~4 of Theorem~1.5 in~\cite{mpr11}).
This proves the main statement in item~\ref{it5} of
Theorem~\ref{structurethm}. Item~3(a) follows from work of
Bernstein and Breiner~\cite{bb2} or Meeks and P\'erez~\cite{mpe3}. Item~3(b) occurs
when the number $k$ of ends of $L_1$ satisfies $2\leq k<\infty$, as follows
from Collin~\cite{col1}. Since $L_1$ has finite genus, then item~3(c) occurs when $k=\infty $ by
Theorem~1 in~\cite{mpr4}. Now the proof is complete.
\end{proof}

\section{The proof of Theorem~\ref{thm1}.}
\label{secproofthm1}
By Collin~\cite{col1} and  L\'opez-Ros~\cite{lor1},
the catenoid is the only properly embedded, connected
genus-zero minimal surface with at least two ends and finite topology,
and so Theorem~\ref{thm1} holds for genus-zero surfaces. Arguing by contradiction,
suppose that for some positive integer $g$, there exists
an infinite sequence $\{ M(n)\} _{n\in \N }$ of
properly embedded minimal surfaces in $\R^3$ of
genus $g$ such that for every $n$, the
number of ends of $M(n)$ is finite and strictly less
than the number of ends of $M(n+1)$ and
the number of ends of $M(1)$ is at least three. By Collin's Theorem~\cite{col1}, each
of these surfaces has finite total curvature with planar and catenoidal ends, where all ends
can be assumed to be horizontal for all $n$ after a suitable rotation.

\subsection{Sketch of the argument.}
The argument to find the desired contradiction to prove Theorem~\ref{thm1} is based on
an inductive procedure, each of whose stages starts by finding a scale of smallest
 non-trivial topology for the sequence of surfaces; next we will identify the limit $L$
of the sequence in this scale (after passing to a subsequence) as being a
properly embedded minimal surface in $\R^3$ with controlled geometry; this control
will allow us to perform a surgery on the surfaces of the sequence (for $n$ sufficiently
large) that simplifies their topology. This simplification of the topology will allow us
to find another smallest scale of non-trivial topology in the subsequent stage of the process
and then to repeat the arguments. The fact that all surfaces in the original sequence have
fixed genus insures that after finitely many stages in the procedure, the simplification
of the topology of the surfaces cannot be by lowering their genus. This fact will be used to prove
that after some stage in the process, all limit surfaces $L$ that we obtain with this
procedure will be catenoids. In turn, this fact will allow us to find compact subdomains
$\Lambda (n)$ with boundary inside the $M(n)$ (for $n$ sufficiently large), that
reproduce almost perfectly formed, large compact pieces of the limit catenoids
suitably rescaled, and we will show that we can find as many of these
subdomains $\Lambda (n)$ as we like and whose associated almost waist circles $\G _n
\subset \Lambda (n)$
separate $M(n)$ and such that these domains form a pairwise disjoint collection.
This separation property of the waist circles will let us control the
flux vector of $M(n)$ along $\G_n$. The final contradiction will follow from an
application the L\'opez-Ros deformation to a suitable non-compact, genus-zero subdomain in
$M(n)$ bounded by two of these separating curves $\G _n$.

\subsection{Rescaling non-trivial topology in the first stage.}
The asymptotic behavior
of $M(n)$ implies that for each $n\in \N$,
there exists a positive number $r_{1,n}$
such that every open ball in $\R^3$ of radius $r_{1,n}$ intersects the surface
$M(n)$ in simply connected components and there is some point $T_{1,n}\in \R^3$ such that
$\overline{\B }(T_{1,n},r_{1,n})$ intersects $M(n)$ in at
least one component that is not simply connected.
Then, the rescaled and translated minimal surfaces
\[
M_{1,n}=\frac{1}{r_{1,n}}(M(n)-T_{1,n})
\]
of finite total curvature have horizontal ends and satisfy the following uniformly locally simply connected
property for all $n\in \N $:
\par
\vspace{.2cm}
\noindent
$(\star )$ Every open ball of radius $1$ intersects $M_{1,n}$ in
disk components, and the closed unit ball $\overline{\B }(1)$ intersects
$M_{1,n}$ in a component $\Omega _n$ that is not simply connected.
\vspace{.2cm}

By property $(\star)$, the compact semi-analytic set $M_{1,n}\cap \overline{\B }(1)$
with $M_{1,n}\cap \partial{\B }(1)$ being analytic,
which admits a triangulation by~\cite{lo1},  contains a
 non-simply connected component $\Omega _n$. It is straightforward to prove that
$\partial \Omega_n$ contains a piecewise-smooth simple closed curve $\G_{n}$ that does not
bound a disk in $M_{1,n}\cap \overline{\B }(1)$; for example, see~\cite{my1}
for similar constructions.
Thus, property $(\star)$ implies the next one:
\par
\vspace{.2cm}
\noindent
$(\star \star)$
There exists a piecewise-smooth simple closed curve $\beta(n)\subset \partial\Omega_n\subset  M_{1,n}\cap\partial\B(1)$
that is not the boundary of a disk in $M_{1,n}$, where $\Omega_n$ is defined in property $(\star)$.

\subsection{Controlling the limit of the rescaled surfaces in the first stage.}
\label{limit1stage}
By the main statement in Theorem~\ref{structurethm} applied to the $M_{1,n}$ with
$W=\varnothing$, after passing to a subsequence, the surfaces in the sequence $\{M_{1,n}\}_n$
converge to a minimal lamination ${\cL}_1$ of $\rth$ with related singular set of
convergence ${S({\cal L}_1)}$.
Assume for the moment that $S({\cal L}_1)=\varnothing$ (this property will be
proven in Lemma~\ref{lemma3} below).
\begin{lemma}
\label{lemanew}
If $S({\cal L}_1)=\varnothing$, then ${\cal L}_1$ consists of a single leaf $L_1$ that
is not simply connected,  has genus at most $g$ and is properly embedded in $\rth$.
\end{lemma}
\begin{proof} By item~\ref{it5} of Theorem~\ref{structurethm}, it suffices to prove that
${\cal L}_1$ contains a non-simply connected leaf.

Reasoning by contradiction, suppose that ${\cal L}_1$ consists entirely of leaves that
are complete minimal surfaces
that are simply connected.  Theorem~\ref{structurethm}
implies that either  ${\cal L}_1$ has a single leaf that is non-flat and properly
embedded in $\rth$ or else
${\cal L}_1$ consists of planar leaves. The uniqueness of the helicoid~\cite{mr7}
demonstrates that if ${\cal L}_1$ has a single leaf $L_1$, then this leaf is a plane or a helicoid.
By  Sard's Theorem, we can take $\de \in (1,2)$ such that
each $M_{1,n}$ intersects the sphere $\partial \B (\de )$ transversely.
Let
$\widetilde{\Omega }_n$ be the  component of $M_{1,n}\cap \overline{\B }(\de )$ that
contains the homotopically non-trivial simple closed curve $\beta(n)\subset \partial \Omega_n$,
where  $\beta(n)$, $ \Omega_n$ are given in property $(\star \star)$.

By the convex hull property for minimal surfaces,
$\widetilde{\Omega }_n$  is a compact minimal surface whose (smooth) boundary lies
in the sphere $\partial \B (\de )$. As the (smooth) limit ${\cal L}_1$ of $\{ M_{1,n}\} _n$
is a collection of planes or it is a helicoid $\cH$ that is a
smooth, multiplicity-one limit of  $\{ M_{1,n}\} _n$,
then for $n$ large $\widetilde{\Omega }_n$ is an almost-flat disk (in the case that  ${\cal L}_1$
is a collection of planes) or $\widetilde{\Omega }_n$ is a disk that is a small normal graph
over its orthogonal projection to the limit helicoid $\cH$. By the convex hull property for minimal surfaces,
the simple closed curve $\beta(n)$ is the boundary  a disk
in $ \Omega _n\subset M_{1,n}$, which is a contradiction that proves the lemma.
\end{proof}

By item~\ref{it5} of Theorem~\ref{structurethm}, then ${\cal L}_1$ consists of a single
leaf $L_1$ which is a properly embedded minimal surface in $\R^3$ with genus at most $g$,
the convergence of $\{ M_{1,n}\} _n\to L_1$ has multiplicity 1 and exactly one of the cases
3(a), 3(b) or 3(c) hold.

\begin{lemma}
\label{lemma2}
If $S({\cal L}_1)=\varnothing$, then case~\ref{it5}(c)
of Theorem~\ref{structurethm} cannot hold.
\end{lemma}
\begin{proof}
Assume that $L_1$ has two limit ends and we will find a contradiction.
In this proof we will use the general description of such a minimal surface that is
given in Theorems~1 in~\cite{mpr3} and 8.1 in~\cite{mpr6},
as well as the terminology in those theorems.
After a fixed rigid motion
$A \colon \R^3\rightarrow \R^3$ we may assume that
\begin{equation}
\label{eq:A1}
\Sigma  = A(L_1)
\end{equation}
has an infinite number of middle ends that are horizontal planar ends.
Furthermore, there is a
representative $E \subset \Sigma  $ for the top limit end
of $\Sigma $ that is conformally $\esf ^1
\times [t_0, \infty )$ punctured in an infinite set of points $\{e_1, e_2, \ldots,
e_k, \ldots\}$ that correspond to the planar ends of $E$; here,
$\esf ^1$ is a circle of circumference equal to the vertical
component of the  flux vector of $E$,
which is the integral of the inward pointing conormal of $E$ along its boundary.
In this conformal representation of $E$, we also have
\[
x_3 (\theta , t)
= t,\quad x_3(e_k) < x_3(e_{k+1})
\ \mbox{ for all $k \in \N $ and }\; \lim_{k
  \rightarrow \infty} x_3(e_k) = \infty.
  \]

Given $i\in \N $, let $t_i = \frac{1}{2}[x_3(e_i) + x_3(e_{i+1})]$.
Consider the simple closed curves
\[
\gamma(i) = x_3{^{-1}}(\{ t_i\} ) \subset E,
\qquad
\g (0)=x_3^{-1}(\{ t_0\} )= \partial E.
\]
 Let $S$ be the closed horizontal slab in $\R^3$ between
the heights $t_0$ and $t_{2g+2}$, where $g$ is the genus of $M_{1,n}$.
Observe that $\Sigma \cap S\subset E$ is a connected minimal surface
 whose boundary consists of the two simple closed planar curves
$\g (0),\g (2g+2)$,
and that $\Sigma \cap S$ has $2g+2$ horizontal planar ends.

Given $R>0$, let ${\cal C}(R) =
\{(x_1, x_2, x_3) \mid x_1^2 + x_2^2 \leq R^2\}$ be the solid cylinder of radius $R$.
For any fixed $\ve > 0$ small, there exists an $R_1$ large such that
$\g (i)$ is contained in ${\cal C}(R_1)$ for $0\leq i\leq 2g+2$, and
$(\Sigma \cap S) \setminus {\cal C}(R_1)$ consists of
$2g+2$ annular graphs which
are $\ve/2$-close  in the $C^2$-norm to the set of planes
\[
{\cal P} = \{x_3^{-1}(\{ x_3(e_1)\} ), \ldots, x_3^{-1}(\{
x_3(e_{2g+2})\} )\} .
\]

We next transport the above structure from the limit surface $\Sigma $ to the sequence of surfaces
\begin{equation}
\label{eq:A2}
\Sigma (n)=A(M_{1,n}), \quad n\in \N,
\end{equation}
that converge smoothly to $\Sigma $ as $n\to \infty $ (observe that {\it a priori,}
the $\Sigma (n)$ might fail to have horizontal ends).

     For every $R_2 > R_1$ there exists an $N=N(R_2)\in \N$
such that for $n \geq N$,
\[
\Sigma _S(n,R_2) = \Sigma(n) \cap S \cap {\cal C}(R_2)
\]
 is $\ve$-close to the planar domain
 $\Sigma _S(R_2) = \Sigma \cap S \cap {\cal C}(R_2)$
in the $C^2$-norm; in particular $\Sigma _S(n,R_2)$ is also a connected
planar domain for all $n\geq N$.
By initially choosing $\ve$ sufficiently small,
the following properties also hold for every $n\geq N$:
\begin{enumerate}[(B1)]
\item $\partial \Sigma _S(n,R_2)$ consists of $2g+4$ simple
closed curves which are arbitrarily close to $\partial \Sigma _S(R_2)$
 if $n$ is sufficiently large.
\item We can approximate the curves $\gamma(i)$  by simple
closed planar curves $\gamma(i,n)$ in $\Sigma _S(n, R_2)\cap
x_3^{-1}(\{ t_i\} )$, for all $i=0,\ldots ,2g+2$.
\item $\partial \Sigma _S(n,R_2)$ has two components $\g(0,n),\g(2g+2,n)$
lying on $\partial S $ and $2g+2$
simple closed curves
$\a _1(n),\ldots $, $\a _{2g+2}(n)\subset \partial {\cal C}(R_2)$,
which are graphs over the circle
$\partial {\cal C}(R_2)\cap \{ x_3=0\} $, ordered by their
relative heights, see Figure~\ref{fig3}.
\end{enumerate}
\begin{figure}
\begin{center}
\includegraphics[width=14.5cm]{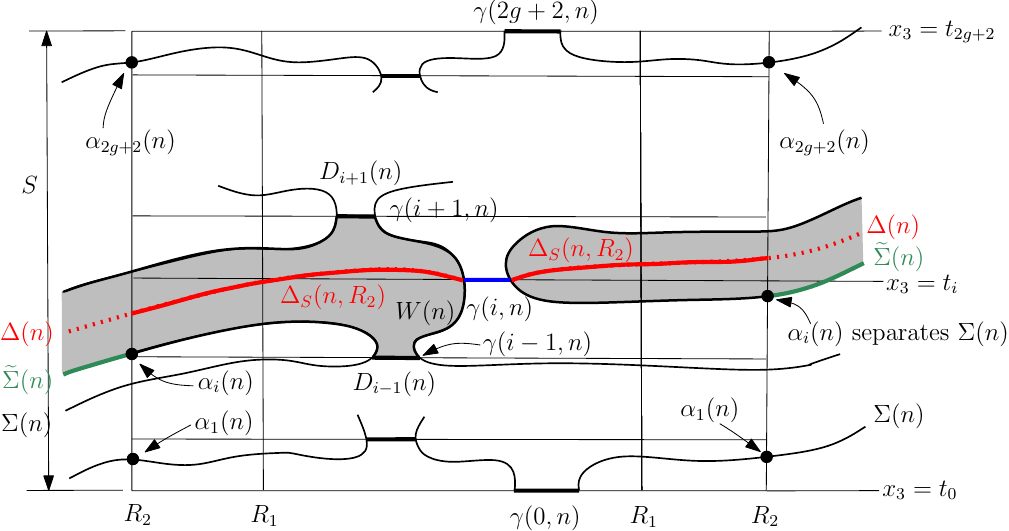}
\caption{Producing the stable minimal surface $\Delta (n)$ in $W$.}
\label{fig3}
\end{center}
\end{figure}

An elementary argument shows that in a compact surface with genus $g$ and empty boundary,
any connected planar subdomain with $2g+4$ boundary
components has at least four boundary components
that separate the surface. Since $\Sigma (n)$
has genus $g$ and $\Sigma _S(n,R_2)$ is a connected
compact planar domain with $2g+4$ boundary components,
there exists an $i\in \{1,\ldots ,2g+1\}$
such that the corresponding curve $\a _i(n)$ separates $\Sigma (n)$.

\begin{assertion}
\label{ass3.2}
The simple closed curve $ \g (i,n)$ separates $\Sigma (n)$.
\end{assertion}
{\it Proof of the assertion.}
Let $\widetilde{\Sigma }(n)$ be the component of $\Sigma (n)\setminus \a _i(n)$ that is disjoint
from $\Sigma _S(n,R_2)$.
  Note that there is a disk in ${\cal C}(R_2)$ bounded by $\a _i(n)$
which only intersects $\Sigma (n)$ along $\a _i(n)\cup \g (i,n)$. The union of this
disk with $\widetilde{\Sigma }(n)$ is a properly embedded, piecewise smooth
surface in $\R^3$. After a
slight perturbation of this surface in a small neighborhood of $\widetilde{\Sigma }(n)$,
we obtain a connected properly embedded surface $\Omega (n)\subset \R^3$ which intersects
$\Sigma (n)$ only along $\g (i,n)$. Since properly embedded surfaces in $\R^3$
separate $\R^3$, we deduce that  $ \g (i,n)$ separates $\Sigma (n)$, which proves the assertion.

Note that we can assume that $\Omega(n)\setminus \g(i,n)$ consists of two components, one of which is
a planar disk contained in $S\cap \cC(R_1)$.
  {\hfill\penalty10000\raisebox{-.09em}{$\Box$}\par\medskip}

Consider the surface $\Omega (n)$ defined in the proof of Assertion~\ref{ass3.2}.
Let $W(n)$ be the closed complement of $\Sigma (n)$ in $\R^3$ that intersects
$\Omega (n)$ in a non-compact connected surface 
with boundary $\g (i,n)$. Denote by
\[
D_{i-1}(n)\subset x_3^{-1}(\{ t_{i-1}\} ),\quad D_{i+1}(n)\subset x_3^{-1}(\{
t_{i+1}\} ),
\]
 the planar disks bounded respectively by $\g (i-1,n),\g (i+1,n)$.
Observe that
  $\partial W(n)\cup D_{i-1}(n)\cup D_{i+1}(n)$ is a
good barrier for solving Plateau problems in the abstract
Riemannian piecewise smooth three-manifold $N(n)$ obtained as the metric completion of the interior of
$W(n)$ (in particular, the planar horizontal disks $D_{i-1}(n),D_{i+1}(n)$ could appear twice in
the boundary of $N(n)$), see~\cite{my2}.
Also note that $\g (i,n)$ separates the boundary of $N(n)$ since
the surface $\Omega (n)\cap N(n)$ separates $N(n)$ and $\partial
[\Omega (n)\cap N(n)]=\g (i,n)$.
Thus, a standard argument
(see e.g., Lemma~4 in Meeks, Simon and Yau~\cite{msy1})
using a compact exhaustion of  the closure of one of the components of $\partial N(n)\setminus \g (i,n)$
 implies that we can find a connected,
orientable, non-compact, properly
embedded least-area surface
\[
\Delta (n)\subset W(n)\setminus [D_{i-1}(n)\cup
D_{i+1}(n)]\subset \rth, \ \mbox{with}\ \ \partial \Delta (n)=\g (i,n).
\]
 By the main result in Fischer-Colbrie~\cite{fi1},
 $\Delta(n)$ has finite total curvature and hence it has
a positive finite number of planar and catenoidal ends that lie in $W(n)$. By the
definition and the uniqueness of the limit tangent plane at infinity~\cite{chm3},
we deduce
that the limiting normal vectors to the ends of $\Delta (n)$ are parallel to
the limiting normal vectors to the ends of $\Sigma(n)$.

Let $\Delta_S(n, R_2)$ be the component of $
 \Delta(n) \cap S \cap {\cal C}(R_2)$ whose boundary contains $\g (i,n)$. Note that the
 other boundary components of $\Delta _S(n,R_2)$
 all lie on $\partial {\cal C}(R_2)\cap W(n)$. For $R_1$ large and $R_2\gg R_1$,
  curvature estimates for stable minimal surfaces~\cite{sc3}
imply that $\Delta _S(n, R_2) \setminus  {\cal C}(R_1)$ consists of almost-horizontal annular graphs,
for $n$ sufficiently large. By the
area-minimizing property of $\Delta (n)$, we have that
for $R_2$ much larger than
$R_1$ and $n$ sufficiently large, there is only one
such almost-horizontal graph, which can be assumed to be oriented by the upward
pointing normal.

Let $\widetilde{\Delta}(n)$ be the closure of $\Delta(n) \setminus  \Delta_S(n,R_2)$. Since
$\Delta (n)$ has finite total curvature, then $\widetilde{\Delta}(n)$ compactifies after
attaching its ends to a compact Riemann surface $\overline{\widetilde{\Delta }(n)}$ with
boundary, and the Gauss map $G_n\colon \widetilde{\Delta }(n)\to \esf^2$ extends
smoothly across the ends to $G_n\colon \overline{\widetilde{\Delta }(n)}\to \esf^2$,
with values at the ends that lie in a pair of antipodal points $\pm a\in \esf^2$.
Observe that $\widetilde{\Delta }(n)$ is strictly stable (because $\Delta (n)$
is stable and $\Delta (n,R_2)$ has positive area), and so, $\overline{\widetilde{\Delta }(n)}$
is also strictly stable.  Since
$\widetilde{\Delta}(n)$ is almost-horizontal along its boundary, then
$G_n(\partial \widetilde{\Delta}(n))$ is contained in a small
neighborhood $Q(n,R_2)$ of $(0, 0,1)$ in $\esf ^2$.
Since the Gaussian image of the ends of
$\widetilde{\Delta}(n)$ is contained in $\{ \pm a\} $ and $G_n$ is either constant or
an open map, then we deduce that either $G_n(\overline{\widetilde{\Delta}(n)})\subset Q(n,R_2)$,
or else the interior of $G_n(\overline{\widetilde{\Delta}(n)})$ contains the horizontal equator $\esf^2\cap \{ z=0\} $
for $n$ large enough. The last possibility contradicts that $\overline{\widetilde{\Delta }(n)}$ is strictly stable,
 as the inner product of $G_n$ with $(0,0,1)$ is a Jacobi function on
 $\overline{\widetilde{\Delta }(n)}$ whose zero set does not intersect
the boundary $\partial \overline{\widetilde{\Delta }(n)}$.
Therefore, $G_n(\widetilde{\Delta}(n))\subset Q(n,R_2)$. Note that as $n$ and $R_2$ approach
$\infty $, $Q(n,R_2)$ limits to be $(0,0,1)$.
Applying Lemma~1.4 in~\cite{mr8}, it follows that $\widetilde{\Delta}(n)$ is
a connected graph over its projection to the $(x_1,x_2)$-plane and the ends of
$\Sigma(n)$ are horizontal.  This implies that the rigid motion $A$
that appears in equations (\ref{eq:A1}), (\ref{eq:A2})
can be taken to be the identity map; in particular, $L_1$
has horizontal limit tangent plane at infinity.

Since $\g (i,n)$ separates $\Sigma (n)$ (by Assertion~\ref{ass3.2}) and the
ends of $\Sigma (n)$ are horizontal, then
$\g (i,n)$ has vertical flux
vector, in the sense that the integral of the unit conormal vector to
$\Sigma (n)$ along $\g (i,n)$ is a vertical vector.
As $\Sigma (n)$ converges
$C^2$ to $\Sigma $ as $n\to \infty $, and $\g (i,n)$ limits to
$\g (i)$, then $\g (i)$ also has vertical flux vector. But Theorem~6 in~\cite{mpr3}
implies that $\Sigma $
does not have vertical flux vector along such a separating curve. This contradiction
finishes the proof of the lemma.
\end{proof}

\begin{lemma}
\label{lemma3}
$S({\cal L}_1) =\varnothing$.
\end{lemma}

\begin{proof}
Reasoning by contradiction, assume that
$S({\cal L}_1)\neq \varnothing$. By item~2 of Theorem~\ref{structurethm},
${\cal L}_1$ is a foliation of $\R^3$ by parallel planes and
$S({\cal L}_1)$ consists of one or two  lines that are orthogonal to ${\cal L}_1$.

Suppose that $S({\cal L}_1)$ consists of a single line. By property $(\star \star)$,
$M_{1,n}\cap \overline{\B }(1)$ contains a piecewise-smooth simple closed curve $\beta(n)$
that does not bound a disk in $M_{1,n}$. We claim that the restriction of the
injectivity radius function $I_{M_{1,n}}$ of $M_{1,n}$
to $\beta(n)$ is bounded from below by some
$I_0>0$ and from above by some $I_1>0$, where both constants do not depend on $n$.
Note that $I_0\geq 1$ since any
open  ball in $\rth$ of radius 1 intersects $M_{1,n} $ in disks
of non-positive Gaussian curvature with their boundaries in
the boundary of the ball.  If $I_1$ fails to exist, then, after replacing by a subsequence,
there exist points $x_n\in \beta(n)$ such that
$I_{M_{1,n}}(x_n)>n$.
In particular,
by Proposition~1.1 in Colding and Minicozzi~\cite{cm35}
 for $n$ sufficiently
large, $x_n$ is contained in a component of
$M_{1,n}\cap \B(x_n,3)$ that is an open disk $D_n$ with boundary in the boundary of $\B(x_n,3)$.
By the triangle inequality,
the boundary of $\B(x_n,3)$ lies outside of $\ov{\B}(1)$. Therefore, the simple closed
curve $\beta(n)\subset D_n \subset M_{1,n}$
is the boundary of a disk in $M_{1,n}$, which is a contradiction.
This contradiction proves the claim that
the values of $I_{M_{1,n}}$ restricted to $\beta(n)$ lie in an interval $[I_0,I_1]$.

A standard argument (see, for instance, Proposition~2.12, Chapter~13 of~\cite{doc2})
produces, for each $n\in \N$, a
 non-trivial geodesic loop $\G _n\subset  M_{1,n}$ parameterized by arc length,
possibly not smooth at $\G _n(0)$, such that $\G _n(0)=
\G _n(2d_n)\in \beta(n)$ (here $2d_n=2I_{M_{1,n}}(\G_n(0))$
is the length of $\G _n$),  and both arcs $\G _n|_{[0,d_n]}$,
$\G _n|_{[d_n, 2d_n]}$ minimize length among curves in $M_{1,n}$ with their extrema.
Hence, the length of $\G _n$ lies in the interval
$[2I_0,2I_1]$.
We claim that the limit set $\mbox{Lim}(\{ \G _n\} _n)$ satisfies
$\mbox{Lim}(\{ \G _n\} _n)\subset S({\cal L}_1)$. If on the contrary there exists a point
$x\in \mbox{Lim}(\{ \G _n\} _n)\setminus S({\cal L}_1)$, then the $\G _n$ converge
locally around $x$ (after extracting a subsequence) to an open straight line segment $\G(x)$
passing through $x$ and contained in one of the planes of the limit foliation ${\cal L}_1$.
As we are assuming that $S({\cal L}_1)$ consists of a single line, then the straight line
that contains $\G(x)$ cannot intersect $S({\cal L}_1)$ at two points. This contradicts that
the length of $\G _n$ is bounded from above by $2I_1$,
and proves our claim
that $\mbox{Lim}(\{ \G _n\} _n)\subset S({\cal L}_1)$. In particular,
there exist points $a_n,b_n\in \G _n$ at intrinsic distance in $M_{1,n}$ bounded away from zero,
such that $|a_n-b_n|\to 0$. This contradicts the  chord arc bound given by Theorem~0.5
of~\cite{cm35}. This contradiction proves
that $S({\cal L}_1)$ cannot  consist of a single line.

Since $S({\cal L}_1)$ consists of two disjoint lines $l_1,l_2$
that are orthogonal to ${\cal L}_1$, then item~2.2 of Theorem~\ref{structurethm}
implies that the pairs of multivalued graphs forming inside the $M_{1,n}$
around $l_1,l_2$  are
oppositely handed.
From this point on, the proof of this lemma is similar to the proof of Lemma~\ref{lemma2}.
We now outline the argument along the lines of the previous proof.

Recall that the Riemann minimal examples $\{ {\cal R}_t\}
_{t>0}$ form a one-parameter family of properly embedded, singly periodic, minimal
planar domains in $\R^3$, each one with infinitely many horizontal planar ends (see
e.g., Section~2 of~\cite{mpr6} for a precise description of these classical surfaces).
To identify the natural limit objects for each of the two ends of this one-parameter family,
one normalizes each ${\cal R}_t$ suitably:
after a certain normalization,
the surfaces ${\cal R}_t$ converge as $t\to 0$ to a vertical catenoid,
while as $t\to \infty$, one can normalize the ${\cal R}_t$
so that under two different sequences of translations,
each translated sequence converges to a vertical helicoid,
with the two forming helicoids inside
the ${\cal R}_t$ for $t$ large being symmetric by
reflection in the vertical plane of symmetry of
${\cal R}_t$; hence these helicoids have opposite handedness.
This last property implies that after shrinking
the ${\cal R}_t$ suitably and taking $t\to \infty $, the ${\cal R}_t$ limit to a
foliation ${\cal F}$ of $\R^3$ by horizontal planes, with
singular set of convergence being two vertical lines. With
our language in item~2 of Theorem~\ref{structurethm}, this limit configuration
is a limiting parking garage structure with two
oppositely oriented columns; see Figure~\ref{ULSC1}.
\begin{figure}
\begin{center}
\includegraphics[width=12cm]{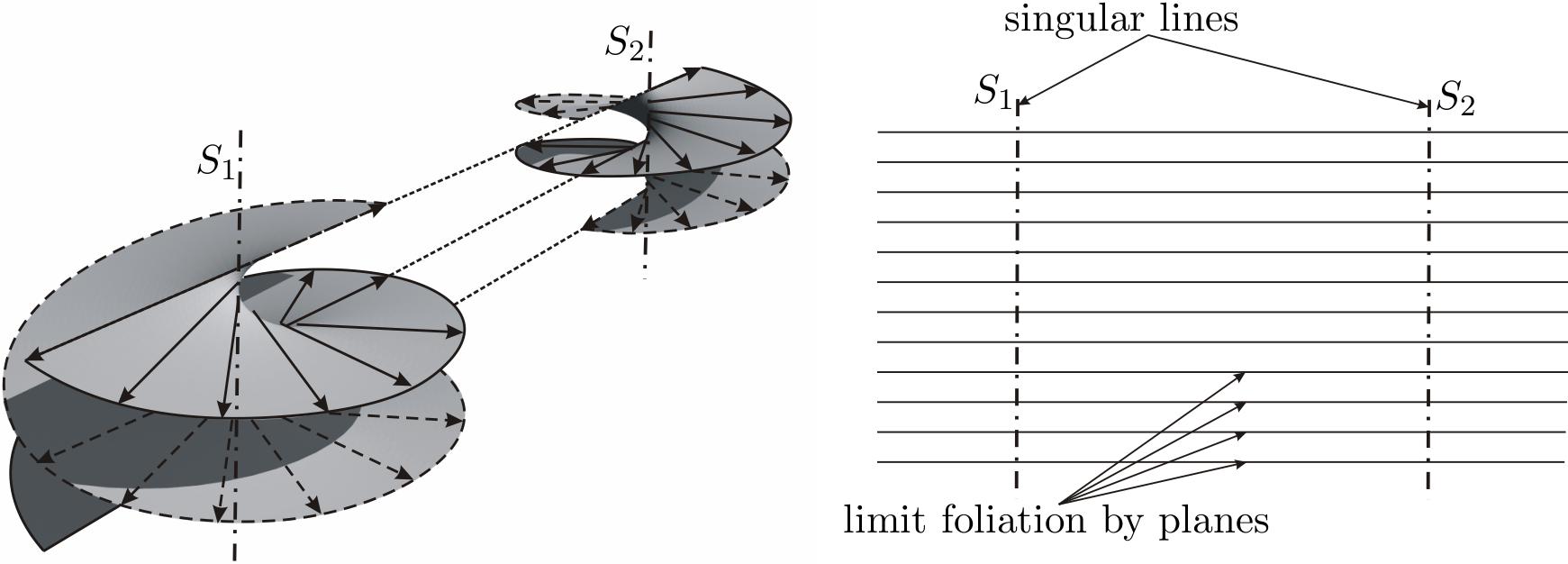}
\caption{$S_1$ and $S_2$ are the columns of the limiting parking garage structure.}
\label{ULSC1}
\end{center}
\end{figure}

In our current setting that the $\{ M_{1,n}\} _n$ converge to the foliation ${\cal L}_1$ except
at two lines $l_1,l_2$ with opposite handedness, the convergence of $\{ M_{1,n}\} _n$ to ${\cal L}_1$ has
the same basic structure as a two-limit-end example. More precisely, let
${\cal C}(R_1)$ be a solid cylinder of radius $R_1$ that contains $S({\cal L}_1)$ in its interior.
Consider the intersection of $M_{1,n}$ with an open slab $S$ bounded by two planes in
${\cal L}_1$. For $R_2>R_1$ and $n$ sufficiently large,
the part of $M_{1,n}\cap S$ in ${\cal C}(R_2)\setminus {\cal C}(R_1)$ contains
an arbitrarily large number of annular graphs
with boundary curves in the cylinders $\partial {\cal C}(R_1), \partial {\cal C}(R_2)$
and these graphs are almost-parallel to the planes in
${\cal L}_1$; note that $R_1$ can be taken large enough
so that the columns of the limiting parking garage structure are contained in $\cC(R_1/2)$.
The remaining components of $M_{1,n}\cap S\cap [{\cal C}(R_2)\setminus {\cal C}(R_1)]$
are graphical almost-horizontal disk components
that are contained in  arbitrarily small neighborhoods of $\partial S$ as $n$ tends to infinity.
Furthermore, inside ${\cal C}(R_1)\cap S$ we can find as many of the related
curves $\gamma(i,n) \subset M_{1,n}$ from the proof of Lemma~\ref{lemma2} as we
desire (take the $\g (i,n)$ as
`connection loops' that converge to a straight line segment joining $l_1,l_2$ and
orthogonal to these lines).
Carrying out the arguments in the proof of the Lemma~\ref{lemma2} we obtain
a contradiction, which proves Lemma~\ref{lemma3}.
\end{proof}

So far we have proven that, after passing to a subsequence,  the surfaces $M_{1,n}$
converge smoothly with multiplicity $1$ to a
connected, properly embedded minimal surface $L_1$ in one of the
following two cases.
\begin{enumerate}[(C1)]
    \item $L_1$ is a one-ended surface with positive genus less
    than or equal to $g$ and $L_1$ is asymptotic to a
    helicoid.
    \item $L_1$ has finite total curvature, genus at most $g$ and at least two ends.
\end{enumerate}

\subsection{Surgery in the first stage.}
\label{surgery1stage}
We start this section with a lemma to be used later.
\begin{lemma}
\label{ass2.5}
  Let $\Sigma \subset \R^3$ be a properly embedded minimal surface of finite genus
  and one end.
Then, for $R>0$ sufficiently large, $\partial \B (R)$ intersects $\Sigma $
transversely in a simple closed curve and $\Sigma \setminus \B (R)$ is an annular end
representative of the unique end of $\Sigma $.
Furthermore, given $\ve>0$, $R$ can be chosen large enough so that this end
representative is $\ve $-close in the $C^2$-norm
to an end of a helicoid.

\end{lemma}
\begin{proof}
Since arbitrarily small perturbations of the square of the distance function of $\Sigma$ to the origin in $\rth$
can be chosen to have only
non-degenerate critical points of index less than 2,  elementary Morse theory
implies that for the first sentence of the lemma to hold
it suffices to prove that for $R>0$ sufficiently large, $\partial \B (R)$ intersects $\Sigma $
transversely. Since $\Sigma $ is a properly embedded minimal surface of finite genus and one end, then $\Sigma $
is asymptotic to a helicoid (Bernstein and Breiner~\cite{bb2} or Meeks and P\'erez~\cite{mpe3}); in particular,
after a rotation in $\R^3$ we may assume that for some $R>0$ large and $\de >0$,
the intersection of $\Sigma \setminus \B (R)$ with the region
\[
C(\de )=\{ (x_1,x_2,x_3)\in \R^3\ | \ x_1^2+x_2^2>\de ^2x_3^2\}
\]
consists of two multivalued graphs over their projections to the $(x_1,x_2)$-plane
with norm of their gradients
less than 1, and given $\ve >0$, $R$ can be taken sufficiently large so that
the norm of the gradients of these multivalued graphs are less than
$\ve $. Furthermore, the analytic description in~\cite{hkp1,mpe3} ensures that
there exists a conformal parameterization $X\colon D(\infty ,\rho)\to \R^3$ of an end
representative of the end of $\Sigma $, with associated Weierstrass data
\begin{equation}
\label{eq:WD}
g(z)=e^{iz+f(z)},\quad dh=dz,
\end{equation}
where $D(\infty ,\rho)=
\{ z\in \C \ | \ |z|\geq \rho )$, $f$ is a holomorphic function in $D(\infty ,\rho)$
that extends across $z=\infty $ with
$f(\infty )=0$, $g$ is the stereographically projected
Gauss map of $X$ and $dh$ denotes its height differential.
This analytic description implies that the following properties hold:
\begin{enumerate}[(D1)]
\item $[\Sigma \setminus \B (R)]\cap C(\de )$ is  transversal to $\partial \B(R')$ for every $R'>R$.
\item $\Sigma \setminus \B (R)$ is transverse to every horizontal plane $\{ x_3=t\} $, for all
$t\in \R $.
\item For $|t|$ sufficiently large, $\be _t:=[\Sigma \setminus \B (R)]\cap \{ x_3=t\} $ consists of a
smooth, proper Jordan arc at distance less than 1 to the axis of
the helicoid $H$ to which $\Sigma $ is asymptotic, and the tangent lines to $\be_t$
are arbitrarily close to the constant value determined by the intersection straight
line $H\cap \{ x_3=t\} $.
\end{enumerate}
Property (D3) ensures that circles in $\{ x_3=t\} $
centered at $(0,0,t)$ and radii $r>R_0$ are transverse to
$\be _t$
for $|t|,R_0$ sufficiently large. This implies that
there exists  $R_1>R_0$ large such that for $R$ sufficiently large,
$\partial \B(R)$ is transverse to
$\Sigma \setminus \{ (x_1,x_2,x_3)\ | \ x_1^2+x_2^2\leq R_1^2\} $.
Finally, transversality of $\partial \B(R)$
to $\Sigma \cap \{ x_1^2+x_2^2\leq R_1^2\} $ for $R$ large follows from the fact that
 under any divergent sequence of vertical translations  of $\Sigma $, a subsequence converges to
some vertical translation of  $H$ and, as $R\to \infty$, the unit normal vectors to $\partial \B(R)$
at points of $\Sigma\cap \partial \B(R)\cap\{ x_1^2+x_2^2\leq R_1^2\} $ are
converging to vertical unit vectors.
This finishes the proof of the
first sentence of the lemma. The second sentence follows from the Weierstrass
representation equation (\ref{eq:WD}).
\end{proof}

Given $R_1>0$, let $L_1(R_1)=L_1\cap \overline{\B }(R_1)$.
By Lemma~\ref{ass2.5}, we can take
$R_1$ sufficiently large so that
in case~{(C1)} above, $L_1\setminus L_1(R_1)$ consists of an annular
end representative of the unique end of  $L_1$,
and this annular end representative is $\ve $-close in the
$C^2$-norm
to an end of a helicoid $H\subset \R^3$
for any fixed $\ve >0$ small. In case (C2), the asymptotic geometry of embedded ends with
finite total curvature allows us to assume that
$L_1\setminus L_1(R_1)$ consists of a finite collection
of annular end representatives of the at least two ends of
$L_1$, each one $\ve $-close in the $C^2$-norm
to the end of a plane or of a catenoid.

We will perform the following replacement of $L_1(R_1)$:
\begin{enumerate}[(E1)]
\item Suppose that $L_1$ is in case~{(C1)} above. Replace $L_1(R_1)$ by
a smooth disk that is a small normal graph over its projection to the helicoid
$H$, and so that the union of this new
piece with $L_1\setminus L_1(R_1)$ produces a smooth, properly embedded
surface $\widetilde{L}_1\subset \R^3$. For
$R_1$ large, this replacement can be made in such a way that
$\widetilde{L}_1$ is $\ve $-close to $H$ in the $C^2$-norm
on compact sets of $\R^3$
for an arbitrarily small $\ve >0$ (after choosing $R_1$ sufficiently large).
\item Now assume $L_1$ is in case~{(C2)}.
Then, for $R_1$ large,  $L_1\setminus L_1(R_1)$ consists of a finite
number $r\geq 2$ of non-compact, annular minimal graphs over the
limit tangent plane at infinity of $L_1$, bounded by $r$ closed
curves which are almost-parallel and logarithmically close in terms
of $R$ to an equator on $\partial \B (R_1)$. Replace
$L_1(R_1)$ by $r$ almost-flat parallel disks contained in $\B (R_1)$
so that the resulting surface, after gluing these disks to
$L_1\setminus L_1(R_1)$, is a possibly disconnected, smooth, properly embedded
surface $\widetilde{L}_1\subset \R^3$. This replacement is made
so that $\widetilde{L}_1\cap \overline{\B }(R_1)$ has arbitrarily small second
fundamental form (after choosing $R_1$ sufficiently large).
\end{enumerate}
In either of the cases (E1), (E2), note that the surface
$\widetilde{L}_1$ is no longer minimal, but it is
minimal outside $\ov{\B }(R_1)$.

We finish this first stage by performing the surgery on the surfaces $M_{1,n}$
that converge to $L_1$.
For $n$ large, the surface $M_{1,n}(R_1)=M_{1,n}\cap \overline{\B }(R_1)$
can be assumed to be
arbitrarily close to $L_1(R_1)$ in the $C^2$-norm. Modify $M_{1,n}$
in $\overline{\B }(R_1)$ as
we did for $L_1$ to obtain a new smooth, properly embedded surface
$\widetilde{M}_{1,n}$ that is $C^2$-close to $\widetilde{L}_1$ in
$\overline{\B }(R_1)$. Observe that if $L_1$ is in case (C1), then
$\widetilde{M}_{1,n}$ is connected and has the same number of ends
as $M_{1,n}$, while if $L_1$ is in case (C2), then $\widetilde{M}_{1,n}$
might fail to be connected, and the total number of ends of the connected
components of $\widetilde{M}_{1,n}$ is equal to the number of ends of
$M_{1,n}$ (see Figure~\ref{fig4} for a topological representation
of this surgery procedure).
\begin{figure}
\begin{center}
\includegraphics[width=13cm]{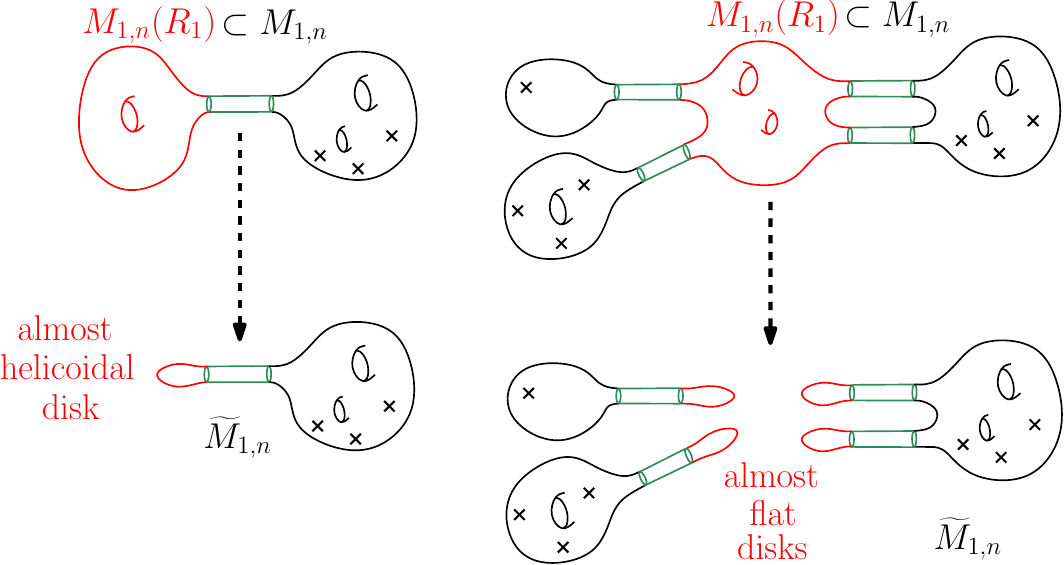}
\caption{Topological representation of the surgery in the first stage, where
the ends are represented by crosses. On the left,
replacement (E1) is done on $M_{1,n}$, thus
the compact surface
$M_{1,n}(R_1)$ (above, red) has only one boundary component and is replaced
by an almost helicoidal disk (below, red); after smoothing along a transition region
(represented by a green annulus), the resulting surface $\widetilde{M}_{1,n}$
is connected and has the same number of ends
as $M_{1,n}$. On the right, replacement (E2) occurs; in this
example, the limit surface $L_1$ has genus 2 and 4 ends,
each end producing a boundary component of
$M_{1,n}(R_1)$ that is glued after the surgery to an almost flat disk. The resulting surface
$\widetilde{M}_{1,n}$ disconnects into three components, with 1, 3 and 3 ends
respectively. Observe that replacement (E1) strictly decreases the genus of $\widetilde{M}_{1,n}$.
}
\label{fig4}
\end{center}
\end{figure}

\subsection{Rescaling by non-trivial topology in the second stage.}
\label{rescaling2stage}
Since the number of ends of $\widetilde{M}_{1,n}$ is unbounded as $n\to \infty $,
then $\widetilde{M}_{1,n}$
has some component that is not simply connected for $n$ large.
As $\widetilde{M}_{1,n}$ has
catenoidal or planar ends, then for $n$ large
 there exists a largest positive number
$r_{2,n}$ such that for every open ball $B$ in $\R^3$ of radius $r_{2,n}$,
every simple closed curve in $\widetilde{M}_{1,n}\cap B$
bounds a disk in $\widetilde{M}_{1,n}$ (but not necessarily inside $B$
as $\widetilde{M}_{1,n}$ does not necessarily satisfy the convex hull property),
and there exists a closed ball of radius $r_{2,n}$
centered at a point $T_{2,n}\in \R^3$ whose
intersection with $\widetilde{M}_{1,n}$ contains a simple closed curve that is
homotopically non-trivial in $\widetilde{M}_{1,n}$.

Since for $n$ large, every component of $\widetilde{M}_{1,n}\cap \B (2R_1)$ is
simply connected and we can assume $R_1\geq
2$, then every simple closed curve homotopically non-trivial on $\widetilde{M}_{1,n}$
that is contained in a ball of radius 2 is necessarily disjoint from $\B (R_1)$.
By the last paragraph, there exists a simple closed curve $\G \subset
\widetilde{M}_{1,n}$ that is homotopically non-trivial in
$\widetilde{M}_{1,n}$ and that is contained in a closed ball
of radius $r_{2,n}$. We next check that $r_{2,n}\geq 1$. It clearly suffices to show that this
inequality holds assuming that $r_{2,n}\leq 2$; in this case, the previous argument
implies that $\B(T_{2,n},r_{2,n})$ does not intersect $\B (R_1)$,
and so, $\G\subset M_{1,n}\cap \B(T_{2,n},r_{2,n})$.  Note that
$\G $ is homotopically non-trivial in $M_{1,n}$. In particular, $r_{2,n}\geq 1$ by
property $(\star )$, as desired.

Let $\widehat{M}_{2,n}=\frac{1}{r_{2,n}}(\widetilde{M}_{1,n}-T_{2,n})$. Thus,
$\overline{\B }(1)$ contains a simple closed curve in $\widehat{M}_{2,n}$ that is
homotopically non-trivial in $\widehat{M}_{2,n}$. Also define
\begin{equation}
\label{M2n}
M_{2,n}=\frac{1}{r_{2,n}}(M_{1,n}-T_{2,n}),\quad B_{1,n}=\frac{1}{r_{2,n}}(\B (R_1)-T_{2,n}).
\end{equation}
Clearly, $\widehat{M}_{2,n}$ is homeomorphic
to $\widetilde{M}_{1,n}$ and has simpler topology than
$M_{2,n}$, in the sense that $\widetilde{M}_{1,n}$ has less generators for its
first homology group than $M_{2,n}$; the simplification of the topology of $M_{2,n}$
giving $\widehat{M}_{2,n}$ as a replacement of a subdomain by disks, only occurs inside the ball $B_{1,n}$.

\subsection{Controlling the limit of the rescaled surfaces in the second stage.}
\begin{lemma}
\label{lemma4}
Let $C\subset \R^3$ be any compact set.  Then, for $n$ large, the ball $B_{1,n}$
is disjoint from~$C$.
\end{lemma}
\begin{proof}
Suppose to the contrary, that after
passing to a subsequence, every $B_{1,n}$ intersects
a compact set $C\subset \R^3$. We first
show that the radii of $B_{1,n}$ go to zero as $n\to \infty $. If not,
and again after taking a subsequence, we can assume that the radius of
$B_{1,n}$ is larger than some $\ve >0$ for every $n\in \N $.
This condition together with the fact that the distance from
$\overline{\B }(1)$ to $B_{1,n}$ is bounded independently of $n$, imply the following
two properties:
\begin{enumerate}[I.]
\item The change of scale (\ref{M2n}) between $M_{1,n}$ and $M_{2,n}$ is essentially
the identity (in the sense that the ratio $1/r_{2,n}$ appearing in (\ref{M2n}) is bounded
from above and below by positive constants independently of $n$, and the translational
part $T_{2,n}$ of (\ref{M2n}) is bounded from above independently of $n$). Observe
that $\wh{M}_{1,n}$ and $\wh{M}_{2,n}$ are related by the same change of scale as
the one in (\ref{M2n}) between $M_{1,n}$ and $M_{2,n}$.
\item There exists $r_0>0$ such that the open ball $B_{1,n}'$ concentric with $B_{1,n}$
of radius $r_0$, contains $\overline{\B }(1)$ for every $n$.
\end{enumerate}
Since $\{ M_{1,n}\} _n$ converges smoothly on arbitrarily large compact subsets of
$\R^3$ to $L_1$ and outside $\B (R_1)$, $L_1$ consists of its annular ends
(here we have possibly applied Lemma~\ref{ass2.5}), we conclude
from properties I, II above that for $n$ large, $B_{1,n}'$ intersects $\widehat{M}_{2,n}$ in disks. This last property
contradicts that $\overline{\B }(1)$ contains a closed curve that is homotopically
 non-trivial in $\widehat{M}_{2,n}$. This contradiction shows that
the radii of the balls $B_{1,n}$ tend to zero as $n\to \infty $, provided that these
balls intersect $C$.

By the previous paragraph, after taking a subsequence we can assume that the
sequence of balls $\{ B_{1,n}\} _n$ converges to a point $p\in C$. To proceed with the proof
of Lemma~\ref{lemma4}, we need the following assertion.
\begin{assertion}
\label{ass3.7}
$\{ M_{2,n}\} _n$ is locally simply connected in $\R^3\setminus \{ p\} $.
\end{assertion}
\begin{proof}
Fix a point $q\in
\R^3\setminus \{ p\} $. Then we can write $|p-q|=d\ve $ for $d\geq 10$, $\ve >0$.
Reasoning by
contradiction, suppose that for $\ve $ arbitrarily
small, for $n$ large, we find a simple closed
curve $\G _n\subset M_{2,n}\cap \B (q,\ve )$ that is homotopically non-trivial in
$M_{2,n}$. Since $\ve $ can be assumed to be less than 1, then $\G _n$ must bound
a disk $D_n$ in $\widehat{M}_{2,n}$. $D_n$ cannot be contained in
$M_{2,n}$ since $\G _n$ is homotopically non-trivial in $M_{2,n}$. Thus,
$D_n$ is not minimal and  must intersect $B_{1,n}$.
Note that $D_n\setminus B_{1,n}$ contains a compact, connected, minimal
planar domain whose boundary intersects each of the spheres $\partial \B (q,\ve )$ and
$\partial B_{1,n}$. An elementary application of the
maximum principle for minimal surfaces shows that there is
no connected minimal surface whose boundary is contained in two
such spheres (pass a suitable catenoid between
the closed balls $\overline{\B }(q,\ve )$ and
$\overline{B_{1,n}}$). Thus, $\{ M_{2,n}\} _n$
is locally simply connected in $\R^3\setminus \{ p\}$.
\end{proof}
We next continue with the proof of Lemma~\ref{lemma4}. Consider the sequence of
compact minimal surfaces with boundary
\[
\left\{ M_{2,n} \cap [\overline{\B }(n) \setminus  \B (p,1/n)] \right\}_n.
\]
As we have already observed, this sequence has locally positive injectivity
radius in $\R^3 \setminus  \{p \}$. By Theorem~\ref{structurethm} applied
to this sequence with $W=\{ p\} $,
there a 
minimal lamination ${\cal L}$ of $\R^3 \setminus  \{p \}$ and a closed subset $S({\cal L})\subset
{\cal L}$ such that $\{ M_{2,n} \cap [\overline{\B }(n) \setminus  \B (p,1/n)] \}_n$
converges $C^{\a }$ (for all $\a \in (0,1)$) to ${\cal L}$ outside the singular
set of convergence $S({\cal L})$. Furthermore, the closure $\overline{\cal L}$
of ${\cal L}$ in $\R^3$ has the structure of a minimal lamination of $\R^3$.

\begin{assertion}
\label{ass3.8}
 $\overline{\cal L}$ consists entirely of planes.
\end{assertion}
\begin{proof}
Reasoning by contradiction,
assume that
$\overline{\lc}$ contains a non-flat leaf.
By item~\ref{it5} of Theorem~\ref{structurethm}, $S({\cal L})=\varnothing$ and
$\overline{\cal L}$ consists of a single leaf $L$, which is
a properly embedded minimal surface in $\R^3$ with genus at most $g$.
Furthermore, the convergence of portions of the surfaces
$M_{2,n} \cap [\overline{\B }(n) \setminus  \B (p,1/n)]$ to $L$ is
of multiplicity 1.

Since for all fixed $\ve > 0$ and for $n$
large the area of $M_{2,n} \cap \B (p, \ve)$
is greater than $\frac{3}{2} \pi \ve^2$ (by the
monotonicity of area formula) and $S({\cal L})=\varnothing$,
then we deduce that $p\in L$. As $L$ is
embedded, the area of $L\cap \B (p,\ve )$
divided by $\ve ^2$ tends to $\pi $ as $\ve \to 0$,
which contradicts that the $M_{2,n}$ converge
smoothly to $L$ in $\B (p,\ve )\setminus \{ p\} $ with multiplicity 1.
This completes the proof of the assertion.
\end{proof}

By Assertion~\ref{ass3.8}, $\overline{\cal L}$ consists entirely of planes, which
after a rotation in $\R^3$ can be assumed
to be horizontal.

Since ${\widehat{M}_{2,n}}\setminus \B (p, 1/4)$ is minimal for
$n$ large and every simple closed curve in
${\widehat{M}_{2,n}}\cap\B (p, 1/4) $
is the boundary of a disk
in ${\widehat{M}_{2,n}}$, we claim that the components of
${\widehat{M}_{2,n}}\cap\B (p,1/4) $
are disks in $\widehat{M}_{2,n}$ when $n$ is large.
To see this claim holds, for $n$ large choose a component 
$\Delta$ of ${\widehat{M}_{2,n}}\cap\B (p,1/4) $
and notice that because $\frac{1}{4}<1$, then $\Delta$  lies in a disk $\cD$ 
in ${\widehat{M}_{2,n}}$
such that the  intersection of $\cD$ with $\rth \setminus B_{1,n}$ is minimal. Hence,
by the convex hull property applied to $\cD\cap [\rth \setminus \B (p, 1/4) ]$,
$\Delta\subset \cD\subset\B (p, 1/4) $. Since  $\widehat{M}_{2,n}\cap B_{1,n}$ consists of
disks, then $\Delta$ is a disk, which proves the claim.

It follows that
${\widehat{M}_{2,n}} \cap \overline{\B }(1)$
contains a homotopically non-trivial simple closed curve,
such that  after an isotopy of such a curve in
$\widehat{M}_{2,n}$, produces another  simple closed curve
$\Gamma_n\subset {\widehat{M}_{2,n}}\cap \overline{\B }(3/2)$
which is homotopically non-trivial in ${\widehat{M}_{2,n}}$ and disjoint from ${\B} (p, 1/4)$.

\begin{assertion}
\label{ass3.9}
The singular set of convergence $S(\lc)$ of the
$M_{2,n} \cap [\overline{\B }(n) \setminus  \B (p,1/n)] $ to ${\cal L}$ is non-empty and intersects $\B(3)$.
\end{assertion}
\begin{proof}
Arguing by contradiction, suppose $S(\lc)\cap \B(3)=\varnothing$. Note that
$\B (p,1/4)$ cannot intersect both of the spheres $\partial \B(3/2)$, $\partial \B(2)$;
in what follows we will assume that  $\B (p,1/4)\cap
\partial \B(3/2)=\varnothing$,
as the argument in the other case is similar.
Since all leaves
of $\overline{\cal L}$ are horizontal planes,
$S(\lc)\cap \overline{\B}(2)=\varnothing $
and that the surfaces
$M_{2,n}\cap [\overline{\B}(3/2)\setminus
\B (p,1/4)]$ are compact,
then for $n$ large, each component of
$M_{2,n}\cap [\overline{\B}(3/2)\setminus
\B (p,1/4)]$ that intersects
$\partial \B (5/4)$ is an almost-horizontal
graph over its projection to the $(x_1, x_2)$-plane and is
either a disk with boundary
in $\partial \B (3/2)$ or a planar domain with one boundary component in $\partial \B (3/2) $
and its other boundary components in $  \partial
\B (p,1/4)$
(here we may assume $\partial \B (p, 1/4)$
is transverse to every such minimal surface).
In particular, for $n$ large, every component of
$\widehat{M}_{2,n} \cap \overline{\B }(3/2)$
is a disk, which contradicts the existence of the curve
$\Gamma_n$.
Thus, $S(\lc)\cap \B(3) \neq \varnothing$.
\end{proof}

By Assertion~\ref{ass3.9} and item~2 of Theorem~\ref{structurethm},
$\overline{\cal L}$ is a foliation of $\R^3$ by (horizontal)
planes and the convergence of the
$M_{2,n} \cap [\overline{\B }(n) \setminus  \B (p,1/n)]$ to $\overline{\cal L}$
has the structure of a horizontal
limiting parking garage structure with one or two columns (vertical lines)
as singular set of convergence.
By Assertion~\ref{ass3.9}, $\overline{S(\lc)}$ intersects
$\B (3)$.

The proof of Lemma~\ref{lemma3} applies to show that
$\overline{S(\lc)}$ does not contain two components (observe that the presence
of the point $p$ does not affect the argument of the proof of Lemma~\ref{lemma3},
as this argument can be done in a horizontal slab far from $p$).
Hence, $\overline{S(\lc)}$ consists of a single
 line.

We now check that the possibility that $\overline{S(\lc)}$
is a single line also leads to a contradiction, which will finish the proof of
Lemma~\ref{lemma4}. Define $\ve >0$ by
\[
\ve =\left\{ \begin{array}{ll}
\min \{ \frac{1}{4},\frac{1}{2}d_{\R^3}(p,\overline{S({\cal L})}) \} & \mbox{if }p\notin \overline{S(\lc)},
\\
\frac{1}{2}& \mbox{if }p\in \overline{S(\lc)}.
\end{array}\right.
\]
Choose $R\geq 3$ sufficiently large so that $p\in \B(R/2)$.
Recall from the paragraph just before Assertion~\ref{ass3.9} that we found
a simple closed curve $\Gamma_n\subset \widehat{M}_{2,n}\cap \overline{\B }(3/2)$
which is homotopically non-trivial in $\widehat{M}_{2,n}$ and disjoint from
$\overline{\B }(p,1/4)$ (in particular, $\G_n\cap \overline{B_{1,n}}=\varnothing$
for $n$ large).
For $n$ large, each of the surfaces $M_{2,n}\cap [\overline{\B }(R) \setminus  \B (p, \ve)]$ contains
a main planar domain component $C_n$, which contains the curve $\Gamma_n$ and
a long, connected, double spiral curve on $\partial \B (R)$.
 The component $C_n$ intersects $\partial \B (p, \ve)$ in either a double spiral curve
when $p \in \overline{S(\lc)}$ or in a large number of
almost-horizontal closed curves when $p \notin S(\lc)$.
It follows that $\widehat{M}_{2,n} \cap \overline{\B }(R)$
consists entirely of disks, which contradicts
the existence of $\Gamma_n$. This contradiction completes the proof of
 Lemma~\ref{lemma4}.
\end{proof}

\begin{lemma}
\label{lemma3.10}
The sequence $\{ M_{2,n}\} _n$ is locally simply connected in
$\R^3$.
\end{lemma}
\begin{proof}
The proof of this lemma is similar to the proof of Assertion~\ref{ass3.7}.
The failure of $\{ M_{2,n}\} _n$ to be locally simply connected at a point $q\in \R^3$
implies for $n$ large the existence of a
disk $D_n\subset \widehat{M}_{2,n}$ whose boundary curve
$\G _n\subset M_{2,n}$ is contained in a ball $\B (q,\ve )$ of small radius $\ve \in (0,1)$
($\G _n$ is homotopically non-trivial in $M_{2,n}$), and such that $D_n$ intersects
$B_{1,n}$.

Since $r_{2,n}\geq 1$ as proven in the paragraph just before (\ref{M2n}),
then the radius of $B_{1,n}$ is less than or equal to $R_1$. As
$B_{1,n}$ leaves every compact set for $n$ large, we deduce from the maximum
principle that there is no connected minimal
surface with boundary contained in $\B (q,\ve )\cap B_{1,n}$
when the distance between the balls
$\B (q,\ve )$, $B_{1,n}$ is sufficiently large. This contradicts the existence of the
minimal planar domain $D_n\setminus B_{1,n}$, and completes the proof of the lemma.
\end{proof}

Applying Theorem~\ref{structurethm} to the sequence $\{ M_{2,n}\} _n$ with
$W=\varnothing$, we conclude that there exists a minimal lamination
 ${\cal L}_2$ of $\R^3$ and a closed subset $S({\cal L}_2)\subset {\cal L}_2$ such that
after passing to a subsequence, the $M_{2,n}$  converge in $\R^3\setminus S({\cal L}_2)$
 to ${\cal L}_2$. Now our previous arguments in Lemmas~\ref{lemanew},
\ref{lemma2} and \ref{lemma3} apply without modifications and complete the proof
of the following proposition.

\begin{proposition}
\label{propos3.11}
After passing to a subsequence, the $M_{2,n}$
converge smoothly with multiplicity $1$ to a
connected, properly embedded minimal surface $L_2\subset \R^3$
that lies in one of the cases
(C1) or (C2) stated at the end of Section~\ref{limit1stage}.
\end{proposition}

\subsection{Surgery in the second stage.}
With the notation of the previous proposition and given  $R_2>0$,
we let $L_2(R_2)=L_2\cap \overline{\B }(R_2)$, where
the radius $R_2$ is chosen large enough so that
$L_2\setminus L_2(R_2)$ consists of annular representatives of the ends of $L_2$.
As we did in replacements (E1), (E2) in Section~\ref{surgery1stage}, we
perform the corresponding
replacement of $L_2(R_2)$ by disks to obtain
a smooth, properly embedded (nonminimal)
surface $\widetilde{L}_2\subset \R^3$.
Since $\{ M_{2,n}\} _n$ converges
smoothly with multiplicity $1$ to $L_2$ and $B_{1,n}$ leaves any
compact set of $\R^3$ for $n$ large enough, then
the sequence $\{ \widehat{M}_{2,n}\} _n$
also converges smoothly with multiplicity~1 to $L_2$.

Finally, replace $\widehat{M}_{2,n}\cap \overline{\B }(R_2)$ in a
similar way to get a new smooth properly embedded surface
$\widetilde{M}_{2,n}\subset \R^3$ which is not
minimal but is $C^2$-close to $\widetilde{L}_2$ in $\overline{\B }(R_2)$.
Note that $\widetilde{M}_{2,n}$ has simpler topology than $M_{2,n}$,
the simplification of topology occurring as two replacements by collections of
disks inside the balls $B_{1,n}$
and $\B (R_2)$. This finishes the second stage in a recursive definition
of properly embedded surfaces obtained as rescalings and disk replacements from the
original surfaces $M(n)$.

\subsection{The inductive process.}
We now proceed inductively to produce the $k$-th stage. After passing to a
subsequence of the original surfaces $M(n)$, we assume that for each $i<k$ the
following properties hold:

\begin{enumerate}[(F1)]
\item There exist largest numbers $r_{i,n}\geq 1$ such that in every open ball
$B\subset \R^3$ of radius $r_{i,n}$, every simple closed curve in $\widetilde{M}_{i-1,n}\cap B$
bounds a disk contained in $\widetilde{M}_{i-1,n}$ (in the case $i=1$
we let $\widetilde{M}_{0,n}$ to be $M(n)$).
\item There exist points $T_{i,n}\in \R^3$ such
that $\widetilde{M}_{i-1,n}\cap \overline{\B }(T_{i,n},r_{i,n})$
contains a simple closed curve that is homotopically non-trivial in $\widetilde{M}_{i-1,n}$.

\item The sequence of surfaces
$M_{i,n}=\frac{1}{r_{i,n}}(M_{i-1,n}-T_{i,n})$  is locally simply connected
in $\R^3$, all being rescaled images of the original surfaces $M(n)$.
\item $\{ M_{i,n}\} _n$ converges smoothly with multiplicity $1$
to a connected, properly embedded minimal surface $L_i\subset \R^3$
that lies in one of the cases (C1) or (C2) stated at the end of Section~\ref{limit1stage}.
\item The surface $\widehat{M}_{i,n}=\frac{1}{r_{i,n}}(\widetilde{M}_{i-1,n}-T_{i,n})$
has simpler topology than $M_{i,n}$: the simplification of the topology of $M_{i,n}$
giving $\widehat{M}_{i,n}$
consists of $i-1$ replacements by collections of disks and these
replacements occur in $i-1$ disjoint balls
that leave each compact set of $\R^3$ as $n\to \infty $.
Furthermore, $M_{i,n},\widehat{M}_{i,n}$
coincide outside such $i-1$ balls.
\item There exists a large number $R_i>0$ such that
$L_i\setminus L_i(R_i)$ consists of annular representatives of the ends of $L_i$.
\item There exists a smooth, properly embedded (nonminimal)
surface $\widetilde{L}_i\subset \R^3$ such that
$\widetilde{L}_i$ coincides with $L_i$ in $\R^3\setminus \B (R_i)$ and
$\widetilde{L}_i\cap \overline{\B }(R_i)$ is either
a disk 
(when $L_i$ is in case (C1)) or a finite number of
almost-flat disks (if $L_i$ is in case (C2)).
\item There exist smooth, properly embedded (nonminimal)
surfaces $\widetilde{M}_{i,n}\subset \R^3$ such that
$\widetilde{M}_{i,n}$ coincides with $\widehat{M}_{i,n}$
in $\R^3\setminus \B (R_i)$ and $\widetilde{M}_{i,n}\cap
\overline{\B }(R_i)$ is arbitrarily $C^2$-close to
$\widetilde{L}_i\cap \overline{\B }(R_i)$.
Note that by (F5) above, the surface
$\widehat{M}_{i,n}\cap \overline{\B }(R_i)$ coincides with
$M_{i,n}\cap \overline{\B }(R_i)$. As a consequence,
$\widetilde{M}_{i,n}$ has simpler topology than
$M_{i,n}$, with the simplification of topology consisting of $i$
replacements by collections of disks,
one of these replacements occurring in $\overline{\B }(R_i)$
and the remaining ones inside $i-1$ balls that
leave each compact set of $\R^3$ as $n\to \infty $. Outside these $i$ balls, $\widetilde{M}_{i,n}$
and $M_{i,n}$ coincide.
\end{enumerate}
We now describe how to define $r_{k,n}, T_{k,n},M_{k,n},L_k,\widehat{M}_{k,n}, R_k,\widetilde{L}_k$ and
$\widetilde{M}_{k,n}$.

We define $r_{k,n}$ to be the largest positive number
such that for every open ball $B\subset \R^3$ of radius
$r_{k,n}$, every simple closed curve in $\widetilde{M}_{k-1,n}\cap B$
bounds a disk in $\widetilde{M}_{k-1,n}$.
Following the arguments in the second paragraph of Section~\ref{rescaling2stage},
one proves that $r_{k,n}\geq 1$. Furthermore,
there exists a closed ball of radius $r_{k,n}$ centered at a
point $T_{k,n}\in \R^3$ whose intersection with
$\widetilde{M}_{k-1,n}$  contains a simple closed curve that
is homotopically non-trivial in $\widetilde{M}_{k-1,n}$.
We denote by
\[
M_{k,n}=\frac{1}{r_{k,n}}(M_{k-1,n}-T_{k,n}),
\quad
\widehat{M}_{k,n}=\frac{1}{r_{k,n}}(\widetilde{M}_{k-1,n}-T_{k,n}).
\]
Hence, $M_{k,n}$ is a rescaled
and translated image of the original surface $M(n)$,
and $\widehat{M}_{k,n}\cap \overline{\B }(1)$ contains a curve
that is homotopically non-trivial
in $\widehat{M}_{k,n}$. Finally, $\widehat{M}_{k,n}$ is obtained
from $M_{k,n}$ after $k-1$ replacements by
collections of disks, one of these replacements occurring inside the ball
\[
B_{k-1,n}=\frac{1}{r_{k,n}}(\B (R_{k-1})-T_{k,n})
\]
and the remaining $k-2$ replacements in pairwise disjoint balls
$\widetilde{B}_{1,n}(k),\ldots ,\widetilde{B}_{k-2,n}(k)$ that are
disjoint from $B_{k-1,n}$ and
where rescaled and translated images of the forming limits
$L_1,\ldots ,L_{k-2}$ are captured (see Figure~\ref{fig5}).
Note that the radius of $B_{k-1,n}$
is $\frac{R_{k-1}}{r_{k,n}}\leq R_{k-1}$, and repeating this
argument we have that the radius
of $\widetilde{B}_{j,n}(k)$ is less than or equal to $R_j$
for each $j=1,\ldots ,k-2$.
\begin{figure}
\begin{center}
\includegraphics[width=8.5cm]{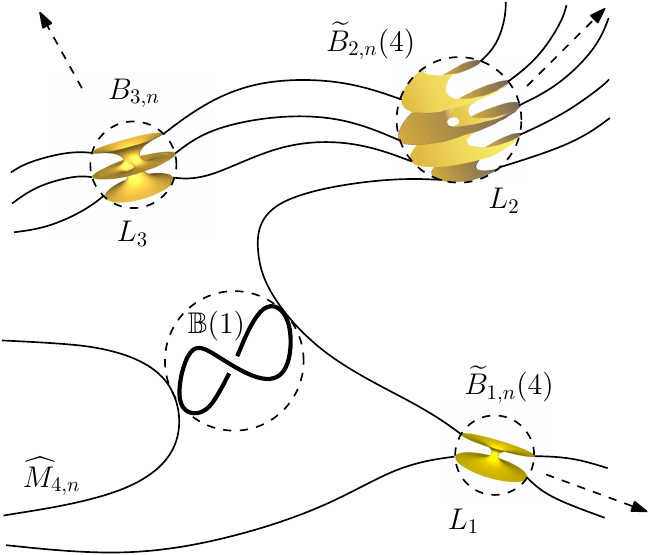}
\caption{The inductive process in the stage $k=4$.} \label{fig5}
\end{center}
\end{figure}

\begin{lemma}
\label{lemma5}
Given any compact set $C\subset \R^3$,  $B_{k-1,n}\cup
\left( \bigcup _{j=1}^{k-2}\widetilde{B}_{j,n}(k)\right) $ is disjoint from~$C$
for $n$ sufficiently large.
\end{lemma}
\begin{proof}
Assume that the lemma fails for some compact set $C$.
The arguments in the first paragraph of the proof of Lemma~\ref{lemma4}
can be adapted to show that if after choosing a subsequence, some of the balls in the collection
$\mathcal{B}=\{B_{k-1,n},\widetilde{B}_{1,n}(k),\ldots ,\widetilde{B}_{k-2,n}(k)\}$ stay at bounded distance from the
origin as $n$ goes to $\infty $, then their corresponding radii go to zero.
This implies that, after extracting a subsequence,
$\bigcup \mathcal{B}=B_{k-1,n}\cup \left( \bigcup _{j=1}^{k-2}\widetilde{B}_{j,n}(k)\right) $
has non-empty limit set as $n\to \infty $ being a finite set of points in $\R^3$,
denoted by $\{ p_1,\ldots ,p_l\} $.

\begin{assertion}
$\{ M_{k,n}\} _n$ is locally simply connected in
$\R^3\setminus \{ p_1,\ldots ,p_l\} $.
\end{assertion}
\label{as313}
\begin{proof}
The proof of the similar
fact in Assertion~\ref{ass3.7} does not work
in this setting, so we will give a different proof.
Arguing by contradiction, we may assume
that there exists a point $p\in \R^3\setminus \{ p_1,\ldots ,p_l\} $ such that for any
$r\in (0,1)$ fixed
and for $n$ sufficiently large, there exists a homotopically non-trivial curve $\g _{k,n}(r)$ in
$M_{k,n}\cap \B (p,r)$. By our normalization, $\g _{k,n}(r)$ bounds a disk $\widehat{D}_{k,n}(r)$
on $\widehat{M}_{k,n}$. Since $\widehat{M}_{k,n}$ and $M_{k,n}$ coincide outside
$\bigcup \mathcal{B}$, we deduce from
the convex hull property that
$\widehat{D}_{k,n}(r)$ must enter some of the balls in $\mathcal{B}$.
Hence, $\widehat{D}_{k,n}(r)$ intersects the boundary of
$\bigcup \mathcal{B}$
in a non-empty collection
${\cal A}$ of curves, each of which is arbitrarily close to a rescaled and translated image
of the intersection of a sphere 
of large radius centered at the origin with
either an embedded minimal end of finite total curvature or with a helicoidal end.
Define
\[
\Omega _{k,n}(r)=\widehat{D}_{k,n}(r)\setminus
 \bigcup \mathcal{B}
\]
which is a planar domain in $M_{k,n}$ whose boundary
consists of $\g _{k,n}(r)$ together with the curves in ${\cal A}$.
Let $\widehat{\Omega }_{k,n}(r)$ be the compact subdomain on $M_{k,n}$ obtained
by gluing to $\Omega _{k,n}(r)$ the forming helicoids with handles whose boundary curves
lie on $\Omega _{k,n}(r)$. Since
 $\widehat{\Omega }_{k,n}(r)$ is a compact minimal surface
with boundary and
$\bigcup \mathcal{B}$
is disjoint from $\B (p,r)$ for $n$ large, then
the convex hull property implies that $\widehat{\Omega }_{k,n}(r)$ has
at least one boundary curve outside $\overline{\B }(p,r)$.
	We will obtain the desired contradiction by
	an argument based on the maximum principle for minimal
	surfaces, applied to $\widehat{\Omega }_{k,n}(r)$ and to
	suitably chosen planes that leave
	$\widehat{\Omega }_{k,n}(r)$ at one side. To find these
	planes, we will first analyze the behavior of
	$\widehat{\Omega }_{k,n}(r)$ near each of its boundary 
components outside $\B(p,r)$. For every component $\G $ of
	$\partial \widehat{\Omega }_{k,n}(r)$ outside
	$\overline{\B }(p,r)$,
the following properties hold.
\begin{enumerate}[(G1)]
\item  $\G $ lies in the boundary of one of the
balls in $\mathcal{B}$,
which we simply denote by $B_{\G }$.
Observe that $\G $ is the intersection of the boundary of $\Omega _{k,n}(r)$ and
the boundary of a compact domain $\Omega'\subset M_{k,n}$ obtained as intersection
of $M_{k,n}$ with the closure of $B_{\G }$. Furthermore, $\Omega'$ is a compact
domain in $M_{k,n}$ where a replacement of type (E2) has been made in a
stage previous to the $k$-stage.

\item There exists a closed neighborhood $U_{\G }$ of $\G $ in
$\widehat{\Omega }_{k,n}(r)$ that lies outside
$B_{\G }$,  an end $\widetilde{E}$ of a vertical catenoid
centered at the origin $\vec{0}\in \R^3 $
or of the plane $\{ x_3=0\} $ and a map $\phi \colon \R^3\to \R^3$, 
which is the composition of a homothety and a rigid motion,
such that $U_{\G }$ can be taken arbitrarily
close to $\phi (E)$ (by taking $n$ large enough and the radii $R_i$, $i<k$, fixed but sufficiently large),
where $E$ is the intersection of
$\widetilde{E}$ with the closed region between two spheres of
large radii centered at $\vec{0}$, so that $\G $
corresponds through $\phi $ to the intersection of $E$
with the inner sphere, see Figure~\ref{fig7}.
\end{enumerate}
\begin{figure}
\begin{center}
\includegraphics[width=7.5cm]{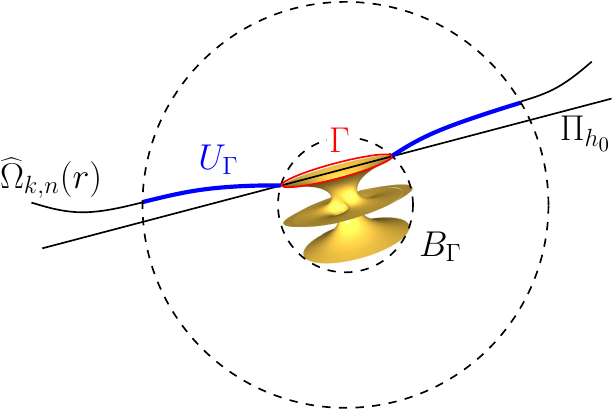}
 \caption{$\widehat{\Omega }_{k,n}(r)$ lies
outside $B_{\G }$ along $\G $.} \label{fig7}
\end{center}
\end{figure}

In particular, the normal vector to $\widehat{\Omega }_{k,n}(r)$ along $\G $ can be assumed to
lie in an arbitrarily small open disk in the unit
sphere, centered at the image by the linear part of
$\phi $ of the limit normal vector to~$\widetilde{E}$.
In the case $\widetilde{E}$ is the end of a catenoid, the
compact subdomain $E$ can be chosen as the intersection of
$\widetilde{E}$ with a  slab of the type
$\{ (x_1,x_2,x_3)\mid a\leq x_3\leq b\} $, for $0<a<b$ large.
For $a$ fixed and $b>a$ arbitrarily large, the sublinearity of the
growth of the third coordinate function on
$\widetilde{E}$ implies that if a plane $\Pi _1\subset \R^3$
touches $E$ at a point of $\{ x_3=a\} $ and leaves $E$ at
one side of $\Pi _1$, then $\Pi _1$ must be arbitrarily close
to the horizontal in terms of $a$. Therefore:
	\begin{enumerate}[($\diamondsuit$)]
	\item If for $n$ large, a plane $\Pi \subset \R^3$ touches $\widehat{\Omega }_{k,n}(r)$
	along $\G $ and leaves $\widehat{\Omega }_{k,n}(r)$
	at one side of $\Pi $, then the orthogonal direction to $\Pi $
	must be arbitrarily close to $\pm \phi (0,0,1)$.
	A similar conclusion holds when $\widetilde{E}$ is the end of the plane $\{ x_3=0\} $.
\end{enumerate}

We now explain how to choose the plane for which
we will use property $(\diamondsuit)$ to find a contradiction based on the maximum principle.
Since the number of components of
$\partial \widehat{\Omega }_{k,n}(r)\setminus \overline{\B }(p,r)$
is bounded independently of $n$, we deduce that the normal
lines to $\widehat{\Omega }_{k,n}(r)$ along
its boundary curves other than $\g _{k,n}(r)$ lie in a collection ${\cal D}_n$
of arbitrarily small (for $n$ sufficiently large) open disks in the projective plane
$\Pe ^2$, the number of which is
bounded independently of $n$. Furthermore, as $n\to \infty $, the ${\cal D}_n$ converge
after extracting a subsequence to a finite set ${\cal D}_{\infty }\subset \Pe ^2$.
Given $n$ large, consider a furthest point $q_n$
in $\partial \widehat{\Omega }_{k,n}(r)$ to $p$. Let $F\in \Pe ^2$ be
the line obtained as a (subsequential) limit of the directions of the position vectors $q_n-p$.
 Suppose for the moment that $F$ does not lie in
${\cal D}_{\infty }$. Consider the family of planes $\{ \Pi _h\subset \R^3\mid h\in \R \} $
orthogonal to $F$, $h$ being the oriented distance to $p$.
If $h_0(n)=\mbox{dist}(p,q_n)$, then
the fact that $q_n$ is a furthest point
	in $\partial \widehat{\Omega }_{k,n}(r)$ to $p$ implies that
$\widehat{\Omega }_{k,n}(r)$ lies entirely at
one side of the plane $\Pi _{h_0(n)}$. Taking $n\to \infty $,
property $(\diamondsuit)$ implies that $F$ lies in ${\cal D}_{\infty }$, which is
contrary to our assumption. Hence $F\in {\cal D}_{\infty }$. Fix
$n$ large and consider the plane $\Pi _{h_0(n)}$ which
has $\widehat{\Omega }_{k,n}(r)$ at one side, with $q_n\in
 \Pi _{h_0(n)}\cap \widehat{\Omega }_{k,n}(r)$. Fix $\ve >0$ small and tilt slightly
 $\Pi _{h_0(n)+\ve }$ to produce a new plane $\Pi'$
  so that $\widehat{\Omega }_{k,n}(r)$ still lies at one side of
 $\Pi'$ and $\Pi'\cap \widehat{\Omega }_{k,n}(r)=\varnothing$. Observe that the
 normal direction to $\Pi'$ can be assumed not to lie in ${\cal D}_{\infty }$.
 Move $\Pi'$ towards $ \widehat{\Omega }_{k,n}(r)$ by parallel planes
 until we find a first contact point (which must lie on $\partial
 \widehat{\Omega }_{k,n}(r)$ by the maximum principle).
 Now we can repeat the argument based on property $(\diamondsuit)$ as before
 with $\Pi'$ instead of $\Pi _{h_0(n)}$ to find
 a contradiction if $n$ is sufficiently large.
This proves the assertion.
\end{proof}

Consider for each $i=1,\ldots ,l$ a ball $B_i$ centered at $p_i$, whose radius is much smaller
than the minimum distance between pairs of distinct points $p_j,p_h$ with $j,h\in \{ 1,\ldots ,l
\} $. To proceed with the proof of Lemma~\ref{lemma5}, we need the following property.
\begin{assertion}
\label{ass3.14}
For each $i$, the components of $\widehat{M}_{k,n}$ in $B_i$ are disks for $n$ large.
\end{assertion}
\begin{proof}
To see this, suppose that for a given $i=1,\ldots ,l$,
there exists a simple closed curve $\g _{k,n}\subset \widehat{M}_{k,n}\cap B_i$
that does not bound a disk on $\widehat{M}_{k,n}\cap B_i$.
As the radius of $B_i$ can be assumed to be less than~$1$, then
$\g _{k,n}$ must bound a disk in $\widehat{M}_{k,n}$.
Now the same proof of Assertion~3.13 
gives a contradiction,
thereby proving Assertion~\ref{ass3.14}.
\end{proof}
We next continue with the proof of Lemma~\ref{lemma5}. For
the fixed value $k$, define the sequence of compact minimal
surfaces
\begin{equation}
\label{eq:seq}
\left\{ M_{k,n} \cap \left[ \overline{\B }(n) \setminus \bigcup _{i=1}^l\B (p_i,1/n)\right]
\right\} _n,
\end{equation}
which has locally positive injectivity radius in $A=\R^3\setminus \{ p_1,\ldots, p_l\} $ by
Assertion~3.13. 
Applying Theorem~\ref{structurethm} to this sequence with $W=
\{ p_1,\ldots ,p_l\} $, we deduce that there exists a
minimal lamination ${\cal L}$ of $\R^3\setminus \{ p_1,\ldots ,p_l\} $ and a closed subset
$S({\cal L})\subset {\cal L}$ such that the sequence defined by (\ref{eq:seq})
converges $C^{\a }$ for all $\a \in (0,1)$ to ${\cal L}$, outside of the singular set
of convergence $S({\cal L})$.
Furthermore, the closure $\overline{\cal L}$ of ${\cal L}$
has the structure of a minimal lamination of $\R^3$.
 In fact, the arguments in the
proof of Assertions~\ref{ass3.8} and~\ref{ass3.9} can be easily adapted to
demonstrate that $\overline{\cal L}$ consists of planes and $S({\cal L})\neq \varnothing$.
Therefore, items~2, 3 of Theorem~\ref{structurethm} imply that $\overline{\cal L}$
is a foliation of $\R^3$ by planes and the convergence of the sequence defined in (\ref{eq:seq})
to $\overline{\cal L}$ has the structure of a limit minimal parking garage structure
with $\overline{S({\cal L})}$ consisting of one or two columns.
Once we know this, the arguments in the proof of Assertion~\ref{ass3.9}
ensure that $\overline{S({\cal L})}$ intersects $\B (3)$.

 Again, the arguments in the
 proof of Lemma~\ref{lemma3} remain valid and imply that $\overline{S(\lc)}$
 consists of a single line. The final step in the proof
 of Lemma~\ref{lemma5} is to discard the case that $\overline{S(\lc)}$ is one line.
 Consider the positive number
 \[
 \ve = \frac{1}{3} \min [\{1\} \cup \{ d_{\R^3}(p_i, p_j)\}_{i \neq j}].
 \]
 Arguing as in the last paragraph of the proof of Assertion~\ref{ass3.9} we conclude that
 there exists a simple closed curve $\Gamma_n\subset \widehat{M}_{k,n}
 \cap \overline{\B }(2)$
which is homotopically non-trivial in $\widehat{M}_{k,n}$ and
disjoint from $\cup _{i=1}^l\overline{\B }(p_i,\ve )$.
Furthermore, for $n$ large each of the surfaces
$M_{k,n}\cap [\overline{\B }(3) \setminus \cup _{i=1}^l \B (p_i,\ve)]$ contains
a main planar domain component $C_n$,
which contains the curve $\G_n$. For $n$ large, $C_n$
intersects $\partial \B (3)$ in a
long, connected, double spiral curve, and it intersects $\partial \B (p_i, \ve)$
in either a double spiral curve  when
$p_i\in \overline{S(\lc)}$ or in a large number of
almost-horizontal closed curves when $p_i \notin \overline{S(\lc)}$, $i=1,\ldots ,l$.
Thus, $\widehat{M}_{k,n} \cap \overline{\B }(3)$
consists of disks, which contradicts the
 existence of $\Gamma_n$. Now the proof of
Lemma~\ref{lemma5} is complete.
\end{proof}

Straightforward modifications in the proof of Lemmas~\ref{lemma3.10}
and Proposition~\ref{propos3.11} give the following lemma, whose proof
we leave to the reader.
\begin{lemma}
The sequence $\{ M_{k,n}\} _n$ is locally simply connected in $\R^3$,
and after passing to a subsequence,
it converges with multiplicity $1$ to a minimal lamination ${\cal L}_k$ of $\R^3$
consisting of a single leaf $L_k$ which satisfies the properties in cases {(C1)} or {(C2)}.
\end{lemma}

We can continue this inductive process indefinitely and using a diagonal
subsequence, we will obtain an infinite sequence $\{ L_k\} _{k\in \N}$
of non-simply connected, properly embedded minimal surfaces,
each one satisfying one of the properties
{(C1)} or {(C2)}. For each fixed $k\in \N$,
$L_k\cap \B (R_k)$ is the limit under homotheties and translations of compact domains
of $M(n)$ that are contained in balls $\widehat{B}_{n,k}$. Moreover,
$\widehat{B}_{n,k}$ is disjoint from $\widehat{B}_{n,k'}$ for $k\neq k'$.

\subsection{The final contradiction.}
Since the genus of $M(n)$ is fixed and finite, for all $k$ sufficiently large
the surface $L_k$ has genus zero (in particular, $L_k$ cannot satisfy case (C1),
see also the caption of Figure~\ref{fig4}).
By the L\' opez-Ros Theorem~\cite{lor1}, $L_k$ is a
catenoid. Fix $k_0$ such that for every $k\geq k_0$, $L_k$ is a catenoid.
Given $k\geq k_0$, the convergence to $L_k$
of suitable homotheties and translations of compact domains of the $M(n)$ contained in the
balls $\widehat{B}_{n,k}$ appearing in the last paragraph ensures that
there exists an integer $n(k)$ such that for all
$n\geq n(k)$, we may assume that $M(n)\cap \widehat{B}_{n,k_0},\ldots ,
M(n)\cap \widehat{B}_{n,k}$ are
close to $k-k_0+1$ catenoids. For these $k,n$ and for any integer $k'$ with $k_0\leq k'\leq k$,
let $\G _{n,k'}$ be the unique simple closed geodesic in $M(n)\cap \widehat{B}_{n,k'}$ which, after
scaling and translation, converges to the waist circle of $L_{k'}$ as $n\to \infty $.

\begin{lemma}
\label{lemma6} For any $m\in \N$, there exists $k\geq k_0$ such that at least $m$ of
the simple closed curves $\G _{n(k),k'}\subset M(n(k))\cap \widehat{B}_{n(k),k'}$
separate $M(n(k))$ where $k_0\leq k'
\leq k$.
\end{lemma}
\begin{proof}
If the lemma were to fail, then for any $k\geq k_0$, there would be a bound on the
number of the curves $\G _{n(k),k^\prime}$ that separate $M(n(k))$.
Since the genus of $M(n(k))$ is independent of $k$, for $k$ sufficiently large
there exist seven of these geodesics $\G _{n(k),k^\prime}$
that bound two consecutive annuli in the conformal
compactification $\overline{M}(n(k))$ of $M(n(k))$.  More precisely, we find
$\Lambda_1,\G_1,\Lambda_2,\G _2,\Lambda_3,\G_3,\Lambda_4$, seven of the
 non-separating curves $\G _{n(k),k^\prime}$, so that
\begin{enumerate}[(H1)]
\item $\Lambda_1\cup \Lambda_4$ is the boundary of a
compact  annulus $A(\Lambda_1, \Lambda_4)\subset \ov{M}(n(k))$
that is the union of compact annuli $A(\Lambda_1, \G_1)$,
$A(\G_1, \Lambda_2)$,
$A(\Lambda_2, \G_2)$,
$A(\G_2, \Lambda_3)$,
$A(\Lambda_3, \G_3)$,
$A(\G_3, \Lambda_4)$ that do not intersect in their interiors.
Observe  by the convex hull property, all of these annuli
contain at least 1 end of $M(n(k))$, as does $M(n(k))\setminus A(\Lambda_1,\Lambda_4)$,
see Figure~\ref{fig7new}.
\end{enumerate}

Let
$A(\G_1,\G_2)$, $A(\G_2,\G_3)$ be the related subannuli in $\overline{M}(n(k))$
bounded by $\G _1\cup \G _2$, $\G _2\cup \G _3$, respectively.

Observe that (H1) now implies:
\begin{enumerate}[(H2)]
\item Each of the
three components of $M(n(k))\setminus (\G _1\cup \G _2\cup \G_3)$ contains at
least two ends of $M(n(k))$.
\end{enumerate}
\begin{figure}
\begin{center}
\includegraphics[width=11.5cm]{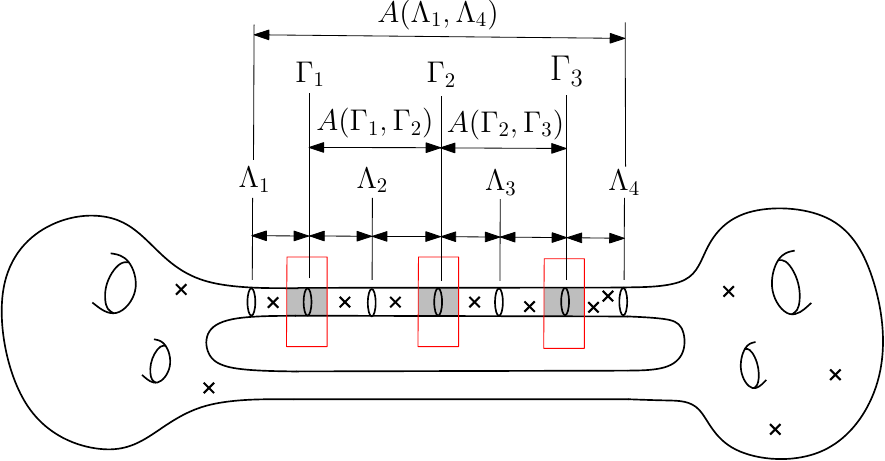}
 \caption{Properties (H1), (H2): Topological representation of $M(n(k))$, with ends represented by
 crosses. When viewed in $\R^3$, the compact portions of $M(n(k))$  enclosed in the red
rectangles represent almost  perfectly formed catenoids, whose almost waist circles
 are the non-separating curves $\G _i$.
 } \label{fig7new}
\end{center}
\end{figure}

 \begin{assertion}
 \label{ass3.17}
  In the above situation,
 $\G _1,\G _2,\G _3$ all bound disks on the same closed complement
of $M(n(k))$ in $\R^3$.
\end{assertion}
\begin{proof}
Suppose that the assertion fails. Then, we may assume without loss of generality that $\G_1$
and $\G_2$ bound disks on opposite sides of $M(n(k))$. Let $W,W_1$ denote
the closures of the two components of $\R^3\setminus M(n(k))$, so that $\G _1$ bounds a disk
$D_1$ in $W_1$ and $\G_2$ bounds another disk $D_2$ in $W$. Observe that
$\G _1$ is not homologous to zero in $W$ (otherwise $\G_1$ would separate
$M(n(k))$, which is contrary to our hypothesis).   Let $\Omega
\subset M(n(k))$ be the (noncompact) planar domain bounded by $\G _1\cup \G_2$.
After a small perturbation of $\Omega \cup D_2$ in $W$ that fixes $\G_1$, we
obtain a new surface $\Sigma $ contained in $W$, such that
$\Sigma \cap M(n(k))=\G _1$. The union of $\Sigma $ together with $D_1$
is a properly embedded surface that intersects
$M(n(k))$ only along $\G _1$.
This implies that $\G_1$ separates $M(n(k))$, which is a contradiction.
\end{proof}

Let $W_1$ be the closure of the component of $\R^3\setminus M(n(k))$ in which
$\G _1,\G _2,\G _3$ all bound disks, which exists by Assertion~\ref{ass3.17}.
Since none of the $\G_1,\G_2,\G_3$ separate
$M(n(k))$, then none of the $\G_1,\G_2,\G_3$
bound properly embedded surfaces in the closure
$W$ of $\R^3\setminus W_1$. As $\G_1\cup \G_2\subset \partial W$ bounds a connected,
 non-compact orientable surface in $W$ (which
is part of $M(n(k))$) and $\partial W$ is a good barrier for
solving Plateau problems in $W$, a standard
argument~\cite{my1,my2} ensures that there exists a
 non-compact, connected, orientable least-area surface
$\Sigma (1,2)\subset W$ with boundary $\partial \Sigma (1,2)=\G _1\cup \G _2$.

\begin{assertion}
\label{ass3.18}
$\Sigma (1,2)$ has just one end, this end has vertical limiting normal vector,
and $\Sigma (1,2)\cap\partial W=\G_1 \cup \G_2$.
\end{assertion}
\begin{proof}
Recall that the simple closed curves $\G _1,\G _2,\G _3$ are the unique closed geodesics in the
intersection of $M(n(k))$ with disjoint balls $B_1,B_2,B_3$ and that $M(n(k))\cap
B_i$ can be assumed to be arbitrarily close to a large region of a catenoid $C_i$ centered
at the center of $B_i$ (and suitably rescaled), $i=1,2,3$.
In order to check that $\Sigma (1,2)$ has exactly one end,
let $X$ be the non-simply connected region of
$B_1\setminus M(n(k))$ that lies between two coaxial cylinders with axis the axis of $C_1$ and radii
$\frac{R}{3},\frac{R}{2}$ where $R$ denotes the radius of $B_1$, see Figure~\ref{fig8}.
\begin{figure}
\begin{center}
\includegraphics[width=6.5cm]{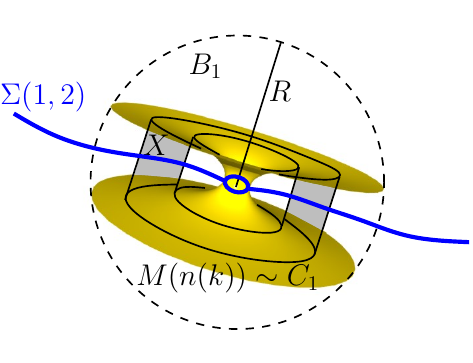}
\caption{The normal line to $\Sigma (1,2)\cap X$ is close to the
limit normal line to $C_1$.} \label{fig8}
\end{center}
\end{figure}

By curvature estimates for stable surfaces, the portion of $\Sigma (1,2)$ contained in $X$
consists of almost-flat graphs parallel to the almost-flat graphs defined by the catenoid
$C$ in the boundary of $X$.  Since the surface $\Sigma (1,2)$ is area-minimizing in
$W$, there is only one such an annular graph. A similar description can be made for
$\Sigma (1,2)$ in the ball $B_2$. After removing the portion of $\Sigma (1,2)$
inside the innermost cylinder in each of these balls, we obtain a connected,
 non-compact, stable minimal surface $\widetilde{\Sigma }(1,2)$ whose Gauss map
$\widetilde{G}\colon \widetilde{\Sigma }(1,2)\to \esf^2$ satisfies that
$\widetilde{G}(\partial \widetilde{\Sigma }(1,2))$ is contained in a small neighborhood of the
limiting normal directions of the corresponding forming catenoids $C_1,C_2$. Since the surface
$\widetilde{\Sigma }(1,2)$ is  stable and connected, it follows that
$\widetilde{G}(\widetilde{\Sigma }(1,2))$ is contained in a small neighborhood $U$
of a point in $\esf^2$
(see a similar argument in the penultimate paragraph of the proof of Lemma~\ref{lemma2}).
In particular, the two forming catenoids in $B_1,B_2$ are
almost-parallel. Since $\widetilde{\Sigma }(1,2)$ lies in the complement of
$M(n(k))$, then the values of $\widetilde{G}$ at the ends of $\widetilde{\Sigma
}(1,2)$ are the North or the South poles, which are contained in $U$.
Thus, the forming
catenoids inside $B_1,B_2$ are approximately vertical and $\widetilde{\Sigma }(1,2)$
is an almost-horizontal graph over its projection to the $(x_1,x_2)$-plane. In
particular, $\Sigma (1,2)$ has exactly one end and this end has vertical limiting normal vector.

Finally, the property that  $\Sigma (1,2)\cap\partial W=\G_1 \cup \G_2$ can be easily
deduced from  the maximum principle and from property (H2) above. This completes
the proof of the assertion.
\end{proof}

Note that $\Sigma (1,2)$ separates $W$ into two regions. Let $W'$ be the closed
complement of $\Sigma (1,2)$ in $W$ that contains $\G_3$ in its boundary. Let
$\G_2'\subset M(n(k))\cap B_2$ be an $\ve $-parallel curve to $\G _2$ in $\partial W'$,
with $\ve >0$ being very small. Since $\G _2'\cup \G _3$ bounds a connected
 non-compact surface in $\partial W'$ (which is part of $M(n(k))$), then
$\G _2'\cup \G _3$ also bounds a connected, non-compact, orientable least-area
surface $\Sigma (2,3)\subset W'$. Note that the arguments in the proof of
Assertion~\ref{ass3.18} apply to give that $\Sigma (2,3)$ intersects $\partial W'$
only along $\G _2'\cup \G _3$, and that
outside $B_2,B_3$,
the surface $\Sigma (2,3)$ is an almost-flat, almost-horizontal
graph over its projection to the $(x_1,x_2)$-plane.

Let $\Sigma $ be the
connected, non-compact, piecewise smooth surface consisting of
\[
\Sigma =\Sigma (1,2)\cup \Sigma
(2,3)\cup D_1\cup D_3\cup A(\G _2,\G _2'),
\]
where for $i=1,3$, $D_i$ is a disk in
$\R^3\setminus W$ bounded by $\G_i$ and $A(\G_2,\G_2')\subset M(n(k))\cap B_2$
is the compact annulus bounded by $\G_2\cup \G_2'$.
Since $\Sigma $ has no boundary and is properly embedded in
$\R^3$, then $\Sigma $ separates $\R^3$ into two open
regions. Let $\Delta (1,3)$ be the connected
component of $M(n(k))\setminus (\G _1\cup \G _3)$
that is disjoint from $\G_2$. Let $W''$ be the connected
component of $\R^3\setminus \Sigma $ that contains $\Delta (1,3)$.
Note that $\Delta (1,3)$ separates $W''$ (this follows because
the piecewise smooth  properly embedded surface
$\Delta (1,3)\cup D_1\cup D_3$ separates $\R^3$, the disks $D_1,D_3$ are contained
in the boundary of $W''$ and $\Delta (1,3)$ is contained in $W''$).
Let $W'''$ be the closed complement
of $\Delta (1,3)$ in $W''$ in which $\G_1$ is not homologous to zero. Let $\G _1',\G
_3' \subset M(n(k))$  be $\ve $-parallel curves to $\G_1,\G _3$
in $\partial W'''$, with $\ve >0$ small.
Since neither $\G_1'$ nor $\G_3'$ separate $M(n(k))$, then
$\G_1'\cup \G _3'$ bounds a connected non-compact, proper
surface in $\partial W'''$ which is part of $M(n(k))$. Thus,
$\G_1'\cup \G _3'$ also bounds a connected, orientable, non-compact,
properly embedded least-area surface
$\Sigma (1,3)\subset W'''$.  As
in Assertion~\ref{ass3.18},
$\Sigma (1,3)$ has exactly one end that is an almost-horizontal graph. The end of
this graph lies between the ends of the two horizontal annular ends of $\Sigma $
since it lies in $W'''\subset W''$, see Figure~\ref{fig6}.
\begin{figure}
\begin{center}
\includegraphics[width=16cm]{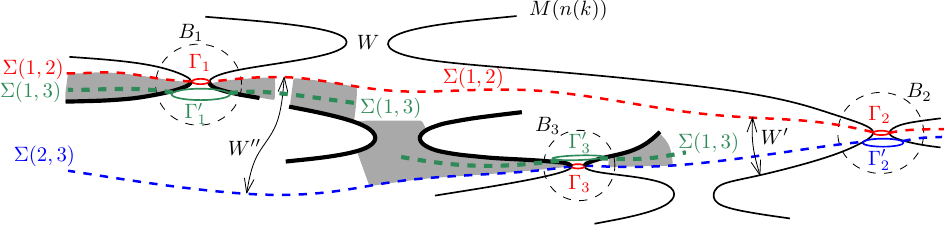}
\end{center}
\par
\vspace{-1cm}
\begin{center}
\includegraphics[width=10cm]{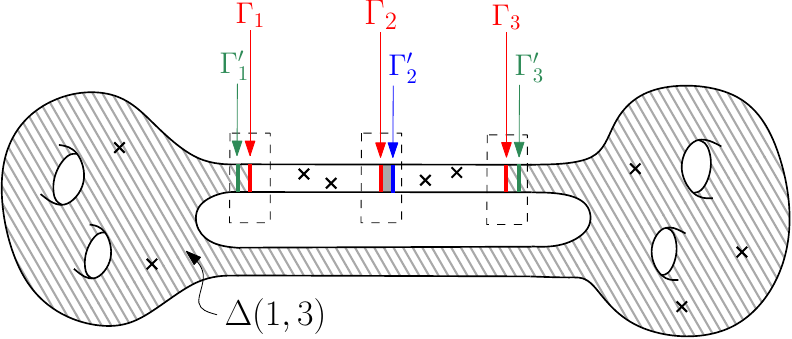}
\caption{Top: Producing a contradiction with three non-separating curves
$\G _1,\G _2,\G _3\subset M(n(k))$. The gray region that contains $\Sigma (1,3)$ is
$W'''$. Bottom: A topological representation of $M(n(k))$, with the curves that appear
in the top figure. The shaded component $\Delta (1,3)$ of $M(n(k))\setminus (\G _1\cup \G _3)$ in the
bottom figure corresponds to the thick black curve in the top figure.
} \label{fig6}
\end{center}
\end{figure}

We now obtain the desired contradiction. Consider the surface
\[
\widetilde{\Sigma}(1,3)=\Sigma (1,3)\cup D_1'\cup D_3',
\]
where $D_i'$ is a disk in $\R^3\setminus W$ bounded
by $\G '_i$, $i=1,3$. The surface $\widetilde{\Sigma }(1,3)$ is properly embedded in $\R^3$
and $\widetilde{\Sigma }(1,3)\cap \Sigma =\varnothing$, hence,
$\Sigma $ must lie on one side of $\widetilde{\Sigma }(1,3)$ in $\R^3$. However,
since the graphical end of $\widetilde{\Sigma }(1,3)$ lies between two graphical
ends of $\Sigma $, we obtain a contradiction that finishes the proof of Lemma~\ref{lemma6}.
\end{proof}

We now complete the proof of Theorem~\ref{thm1}.
By Lemma~\ref{lemma6}, for any $m$ there exists $k\geq k_0$ such that at least
$m$ of the $k-k_0+1$ closed geodesics of the type
$\G _{n(k),k'}\subset M(n(k))\cap \widehat{B}_{n(k),k'}$
separate $M(n(k))$, $k_0\leq k'\leq k$. These $m$
separating curves $\G _{n(k),k'}$ can be assumed to be
arbitrarily close to the waist circles of suitable rescaled,
large compact regions of $m$ disjoint catenoids.
In particular, $M(n(k))$ has non-zero flux vector
along any of these
curves, and the separating property implies that
such flux vectors are all vertical (any separating curve in $M(n(k))$ with non-zero
flux must be homologous to a finite positive number of ends of $M(n(k))$, which have
vertical flux). Therefore, the $m$ forming catenoids inside $M(n(k))$ are all vertical.
Now exchange the geodesics $\G _{n(k),k'}$ by planar
horizontal convex curves $\widetilde{\G}_{n(k),k'}$
in $M(n(k))$, which can be chosen arbitrarily close
to the corresponding $\G _{n(k),k'}$. Since the
genus of $M(n(k))$ is fixed and finite, we can take $m$ large enough so that at
least two of these planar curves, say
$\widetilde{\G}_1,\widetilde{\G}_2$, bound a non-compact planar domain
$\Omega $ inside $M(n(k))$ and bound planar horizontal disks in the same complement
of $M(n(k))$ in $\R^3$. Since $\Omega $ has vertical
catenoidal and/or planar ends, the L\' opez-Ros
deformation~\cite{lor1,ros5} applies to give the desired contradiction. This
finishes the proof of the theorem.

\center{William H. Meeks, III at profmeeks@gmail.com\\
Mathematics Department, University of Massachusetts, Amherst, MA 01003}
\center{Joaqu\'\i n P\'{e}rez at jperez@ugr.es\qquad \qquad Antonio Ros at aros@ugr.es \\
Department of Geometry and Topology and Institute of Mathematics
(IEMath-GR), University of Granada, 18071 Granada, Spain}

\bibliographystyle{plain}

\bibliography{bill}
\end{document}